\documentclass[a4paper,11pt,reqno]{amsart}
\usepackage{amssymb}        \usepackage{amsmath}
\usepackage{latexsym}       \usepackage{amsthm}
\usepackage{amsgen}         \usepackage{amscd}
\usepackage{mathtools}
\usepackage{amstext}        \usepackage{enumerate}
\usepackage{mathrsfs}

\newtheorem{theorem}{Theorem}[section]
\newtheorem{proposition}[theorem]{Proposition}
\newtheorem{lemma}[theorem]{Lemma}
\newtheorem{corollary}[theorem]{Corollary}

\theoremstyle{definition}

\newtheorem*{remark}{Remark}

\newtheorem*{notation}{Notation}
\newtheorem*{acknowledgment}{Acknowledgment}

\numberwithin{equation}{section}

\newcommand{\Z}{\mathbb{Z}}
\newcommand{\Q}{\mathbb{Q}}
\newcommand{\R}{\mathbb{R}}
\newcommand{\C}{\mathbb{C}}
\newcommand{\A}{\mathbb{A}}
\newcommand{\F}{\mathbb{F}}

\newcommand{\G}{\varGamma}
\newcommand{\g}{\gamma}
\newcommand{\w}{\omega}

\newcommand{\la}{\lambda}
\newcommand{\ve}{\varepsilon}

\DeclareMathOperator{\Tr}{Tr}
\DeclareMathOperator{\re}{Re}
\DeclareMathOperator{\im}{Im}

\DeclareMathOperator{\Sym}{Sym}

\DeclareMathOperator{\diag}{diag}
\DeclareMathOperator{\rank}{rank}
\DeclareMathOperator{\ord}{ord}

\newcommand{\Sp}{\mathit{Sp}}
\newcommand{\GL}{\mathit{GL}}
\newcommand{\SL}{\mathit{SL}}

\newcommand{\GSp}{\mathit{GSp}}

\newcommand{\bff}{\mathbf{f}}

\newcommand{\sikakubox}{\leavevmode
  \hbox to.77778em{%
  \hfil\vrule
  \vbox to.775em{\hrule width.25em\vfil\hrule}%
  \vrule\hfil}}

\newenvironment{smatrix}{\left( \begin{smallmatrix}}{\end{smallmatrix} \right)}

\newcommand{\tp}{{\vrule width0pt height8pt}^t}
\newcommand{\ha}{\mathbb{H}}

\renewcommand{\F}{\mathbb{F}}

\newcommand{\epi}{\mathbf{e}}

\newcommand{\1}{\mathbf{1}}
\newcommand{\bfi}{\mathbf{i}}
\newcommand{\itO}{\mathit{O}}

\title[functional equations of Siegel Eisenstein series of level $p$]{On the functional equations of Siegel Eisenstein series of an odd prime level $p$}
\author{Keiichi Gunji}
\date{}

\email{keiichi.gunji@it-chiba.ac.jp}
\address{Department of Mathematics, Chiba Institute of Technology,2-1-1 Shibazono, Narashino, Chiba, 275-0023, Japan}

\begin{document}
\maketitle

\begin{abstract} Let $p$ be an odd prime. In this paper we write down the functional equations of Siegel Eisenstein series of degree $n$, level $p$ with quadratic or trivial characters. First we study the $U(p)$ action in the space of Siegel Eisenstein series to get $U(p)$-eigen functions. Secondly we show that the functional equations for $U(p)$-eigen Eisenstein series are quite simple and easy to write down. Consequently the matrix that represents functional equations of Siegel Eisenstein series are given by the products of simple matrices.
\end{abstract}

\section{Introduction} \label{section_intro}

Functional equations are one of the most important result in the theory of Eisenstein series. The general theories for functional equations in terms of the automorphic representations are already established by Langlands (\cite{La}).

In terms of the classical notation, first Kaufold got the result in degree $2$ case(\cite{Kau}). Kalinin gave the functional equation for any degree case, just rewriting the Langlands or Harish-Chandra's results in classical terminologies (\cite{Kal}). 
After that Mizumoto (\cite{Miz}) gave another proof for functional equations by computing the constant term of the Fourier expansions and unfolding methods.

We shall introduce their results. Let $\G^n= \Sp(n,\Z)$ be the symplectic modular group of rank $n$, matrix-size $2n$ and $P_{0,n}(\Z)$ be its subgroup consisting of the elements whose lower left $(n \times n)$-blocks are zero matrices. For $s \in \C$, an even integer $k$ and $Z = X + \sqrt{-1}Y \in \ha_n$ (Siegel upper half space) we define
\[ E_k^{n}(Z,s) = \det(Y)^{s-k/2} \sum_{\begin{smatrix} * & * \\ C & D \end{smatrix} \in P_{0,n}(\Z) \backslash \G^n} \det(CZ+D)^{-k} | \det(CZ+D)|^{-2s+k}. \]
The right-hand side converges absolutely for $ \re(2s) > n+1$ and it continues meromorphically to whole $s$-plane. 
We put
\[ \mathbb{E}_k^n(Z,s) := \frac{\Gamma_n \left(s+\dfrac{k}{2} \right)}{\Gamma_n(s)} \xi(2s) \prod_{j=1}^{[n/2]} \xi(4s-2j) E_k^{n}(Z,s), \]
here $\Gamma_n(s) = \pi^{n(n-1)/4} \prod_{i=0}^{n-1} \Gamma(s-i/2)$ and $\xi(s) = \pi^{-s/2} \Gamma(s/2) \zeta(s)$ is the completed Riemann zeta function, that satisfies $\xi(s) = \xi(1-s)$. 

\begin{theorem}[{\cite[Theorem 2]{Kal}, \cite[Corollary 6.6]{Miz}}] \label{thm_FE_Siegel_Eisenstein}
Let $\kappa_n = (n+1)/2$. Under the above notations, the functional equation
\[ \mathbb{E}_k^n(Z,s) = \mathbb{E}_k^n ( Z, \kappa_n-s ) \]
holds.
\end{theorem}

The aim of this paper is to consider functional equations for level $p$ case, with an odd prime number $p$. Let
\[ \G_0^n(p) = \left\{  \begin{pmatrix} A & B \\ C  & D \end{pmatrix} \in \Sp(n,\Z) \biggm| C \equiv 0 \bmod p \right\}. \]
As is well-known that the space of Siegel Eisenstein series of degree $n$ with respect to $\G^n_0(p)$ is $(n+1)$-dimensional, thus our functional equations are written in vector valued form. In elliptic modular case, i.e.\ the case of $n=1$, the explicit forms of functional equations are already known (cf.\ \cite[Theorem 4.4.2]{Kub}), but in the higher degree case it is predicted to be quite complicated.

Our basic strategy is investigating  $U(p)$-operator on the space of Siegel Eisenstein series. We explain it more precisely. Let $\psi = \chi_0$ or $\chi_p$ be the Dirichlet character modulo $p$, here $\chi_0$ stands for the trivial character and $\chi_p$ stands for the quadratic character. Let $k$ be an integer such that $\psi(-1) = (-1)^k$.
We can take $\{ w_j \}_{0 \le j \le n}$ as a representative set of the double co-set $P_{0,n}(\Z) \backslash \G^n / \G^n_0(p)$, where $w_j$ is given in (\ref{eq_def_w_j}). Then the Siegel Eisenstein series corresponding to $w_j$ is defined by
\[ E^n_{k,\psi}(w_j;Z,s) = \det(Y)^{s-k/2}  \sum_{\g} \psi^*(\g) \det(CZ+D)^{-k} |\det (CZ+D)|^{-2s+k},
\]
here $\g = \begin{smatrix} A & B \\ C & D \end{smatrix}$ runs through the set $P_{0,n}(\Z) \backslash P_{0,n}(\Z) w_j  \G_0^n(p)$, $\psi^*$ is defined in (\ref{eq_psi*}). It also converges for $\re(2s)>n+1$ and continues meromorphically to whole $s$-plane.
We set $\mathcal{E}_{k,s}^n(\G^n_0(p),\psi)$  the $\C$-vector space spanned by $\{ E_{k,\psi}^n(w_j; Z,s) \mid  0 \le j \le n \}$. Then the eigenvalues of  $U(p)$-operator on $\mathcal{E}^n_{k,s}(\G^n_0(p),\psi)$ are given by $\{ p^{l(s,j)} \mid 0 \le j \le n \}$ with $l(s,j) = n(k/2-s) + 2sj - j(j+1)/2$ (Corollary \ref{cor_eigenvalues_U(p)}). Let $\kappa_n = (n+1)/2$. By an easy calculation one can find
\[ l(s, j) = l(\kappa_n-s, n-j), \]
thus we see the following observation. Let $E^{n,(\nu)}_{k,\psi}(Z,s) \in \mathcal{E}^n_{k,s}(\G^n_0(p),\psi)$ be the suitably normalized $U(p)$-eigenfunction with the eigenvalue $p^{l(s,\nu)}$; it is defined in Proposition \ref{prop_eigen_Eisenstein_series}, explicitly. Then we may expect the existence of the functional equation of the form
\begin{equation*} E^{n,(\nu)}_{n,\psi}(Z,s) = C(s) E^{n,(n-\nu)}_{n,\psi} \left(Z, \kappa_n-s \right) 
\end{equation*}
with some function $C(s)$ in $s$. In the case of $\psi = \chi_0$, it actually follows from Theorem \ref{thm_FE_Siegel_Eisenstein}. On the other hand the case of $\psi= \chi_p$, we can prove it by using the general theory of Langlands. Now we state our  main theorem. Let $\delta_p = 0$ or $1$ according as $p \equiv 1$ or $p \equiv 3 \bmod 4$. We define
\begin{align*} \mathbb{E}_{k,\chi_0}^{n,(\nu)}(Z,s) &  = \g_\nu(\chi_0,s) \dfrac{\Gamma_n\left(s + \dfrac{k}{2} \right)}{\Gamma_n(s)} \, \xi(2s) \prod_{j=1}^{[n/2]} \xi(4s-2j) \,  E_{k,\chi_0}^{n,(\nu)}(Z,s), \\
 \mathbb{E}_{k,\chi_p}^{n,(\nu)}(Z,s)  & = \g_\nu(\chi_p,s) \dfrac{\Gamma_n\left(s + \dfrac{k}{2} \right)}{\Gamma_n \left( s + \dfrac{\delta_p}{2} \right)} \, \xi(\chi_p, 2s) \prod_{j=1}^{[n/2]} \xi(4s-2j) \,  E_{k,\chi_p}^{n,(\nu)}(Z,s),
\end{align*}
here,
\[  \xi(\chi_p,s) = \left( \frac{p}{\pi} \right)^{(s+\delta_p)/2} \Gamma \left( \frac{s+\delta_p}{2} \right) L(\chi_p,s) \]
is the completed Dirichlet $L$-function, that satisfies $\xi(\chi_p,1-s) = \xi(\chi_p,s)$; the factor $\gamma_\nu(\psi,s)$ is defined by
\begin{align*} \g_\nu(\chi_0,s) & =  \frac{1-p^{-2s}}{1-p^{\nu-2s}} \prod_{i=1}^{[\nu/2]} \frac{1-p^{2i-4s}}{1-p^{2\nu+1 -2i-4s}}, \\
 \g_\nu(\chi_p,s) & =  (\ve_p p^{-1/2})^\nu  \prod_{i=1}^{[\nu/2]} \frac{1-p^{2i-4s}}{1-p^{2\nu+1 -2i-4s}},
\end{align*}
with $\ve_p = 1$ or $\sqrt{-1}$ according as $p \equiv 1$ or $p \equiv 3 \bmod 4$.

\begin{theorem}[{Theorem \ref{thm_FE_eigen_function_chi_0}, Theorem \ref{thm_FE_eigen_function_chi_p}}]
We have the functional equation
\[ \mathbb{E}^{n,(\nu)}_{k,\psi}(Z,s) = \mathbb{E}^{n,(n-\nu)}_{k,\psi}(Z, \kappa_n-s). \]

\end{theorem}

In terms of the basis $\{E^n_{k,\psi}(w_\nu;Z,s) \}$ of $\mathcal{E}^n_{k,s}(\G^n_0(p),\psi)$, the functional equations are written as follows. Set
\begin{align*} \mathbb{E}_{k,\chi_0}^{n}(w_\nu;Z,s) &  = \dfrac{\Gamma_n\left(s + \dfrac{k}{2} \right)}{\Gamma_n(s)} \, \xi(2s) \prod_{j=1}^{[n/2]} \xi(4s-2j) \,  E_{k,\chi_0}^{n}(w_\nu;Z,s), \\
 \mathbb{E}_{k,\chi_p}^{n}(w_\nu; Z,s)  & = \dfrac{\Gamma_n\left(s + \dfrac{k}{2} \right)}{\Gamma_n \left( s + \dfrac{\delta_p}{2} \right)} \, \xi(\chi_p,2s) \prod_{j=1}^{[n/2]} \xi(4s-2j) \,  E_{k,\chi_p}^{n}(w_\nu;Z,s).
\end{align*}
Let $B_\psi(s)$ be the representative matrix of the eigenfunctions $E^{n,(\nu)}_{k,\psi}(Z,s)$ using the basis $E^{n}_{k,\psi}(w_j;Z,s)$, that is defined in (\ref{eq_b_psi}). Note that $B_\psi(s)$ is strict upper triangular.
We denote by $T_\psi(s) = ( t_{ij})_{0 \le i,j \le n}$ the anti-diagonal matrix of size $n+1$, whose $(n-\nu,\nu)$-component $(0 \le \nu \le n)$ is given by $\g_\nu(\psi,s) \g_{n-\nu}(\psi, \kappa_n-s)^{-1}$. 
\begin{theorem}[Theorem \ref{thm_FE_chi_0_standard_basis}, Theorem \ref{thm_FE_chi_p_standard_basis}]
We have the functional equation
\[ \begin{pmatrix} \mathbb{E}^{n}_{k,\psi}(w_0;Z,\kappa_n-s) \\ \vdots \\ \mathbb{E}^{n}_{k,\psi}(w_n;Z,\kappa_n-s) \end{pmatrix} = B_\psi(\kappa_n-s)^{-1}T_\psi(s) B_\psi(s) \begin{pmatrix} \mathbb{E}^{n}_{k,\psi}(w_0;Z,s) \\ \vdots \\ \mathbb{E}^{n}_{k,\psi}(w_n;Z,s) \end{pmatrix}.\]
\end{theorem}
Thus we can write the representative matrix of the functional equations by the product of upper triangular matrices and the anti-diagonal matrix.

We explain the contents of this paper. In section \ref{section_U(p)-operator} we study the action of $U(p)$-operator on the space of Siegel Eisenstein series. The results are already known by Walling (\cite{Wal}), however we use the different methods from \cite{Wal}. We can regard the space of modular forms as functions in $\GSp(\A)$ in a standard manner, then the problem becomes to study the action of  $U(p)$-operator on the space of degenerate principal series. This action is purely $p$-local and simple calculations induce our results. In section \ref{section_Fourier_exp} we study the Fourier expansions of Siegel Eisenstein series. The constant teams play crucial roles in studying functional equations. In section \ref{section_functional_eq} we prove our main results.
Finally in section \ref{section_example} we compute the example in the case of degree $2$ with the quadratic character.

\begin{acknowledgment}
The author would like to express his gratitude to Professor Masao Tsuzuki in Sophia University and Professor Tamotsu Ikeda in Kyoto University for their useful comments and a lot of discussions about this topics. 
\end{acknowledgment}

\begin{notation} We denote by $\1_n$ the identity matrix of size $n$. We put $w_n = \begin{pmatrix} 0 & -\1_n \\ \1_n & 0 \end{pmatrix}$. For a commutative ring $R$, we define
\begin{align*}  \GSp(n,R)  & = \{ g \in \GL(2n,R) \mid \tp \! g w_n g = \mu(g) w_n,\ \mu(g) \in R^\times \}, \\
  \Sp(n,R) & = \{ g \in \GSp(n,R) \mid \mu(g) = 1 \}.
\end{align*}
For $g \in \GSp(n,R)$, $(n \times n)$-matrices $A_g$, $B_g$, $C_g$ and $D_g$ are defined by $g = \begin{pmatrix} A_g & B_g \\ C_g & D_g \end{pmatrix}$. We put $\GSp^+(n,\R) = \{ g \in \GSp(n,\R) \mid \mu(g) > 0 \}$.

Let $\ha_n$ be the Siegel upper half space. We usually write the element of $\ha_n$ by $Z = X+\sqrt{-1}Y$. The group $\GSp^+(n,\R)$ acts on $\ha_n$ by 
\[ g \langle Z \rangle = (A_g Z+B_g)(C_g Z + D_g)^{-1} \quad \text{for $Z \in \ha_n$ and $g \in \GSp^+(n,\R)$.} \]
We define the automorphic factor $j(g,Z)$ for $g \in \GSp^+(n,\R)$ and $Z \in \ha_n$ by
\[ j(g,Z) = \mu(g)^{-n/2} \det(C_g Z + D_g), \]
and for the function $f$ on $\ha_n$ and $k \in \Z$ we put
\[ f|_k g(Z) = j(g,Z)^{-k} f( g\langle Z \rangle). \]
Note that for a scalar matrix $a \1_{2n}$ with $a>0$, we have $j(a \1_{2n},Z)=1$. 

Throughout this paper $p$ stands for a fixed odd prime number. An arbitrary prime number is usually denoted by $q$. We put
\[ \G^n_0(p) = \{ g \in \Sp(n,\Z) \mid C_\g \equiv 0 \bmod p \}. \]
For a Dirichlet character $\psi$ modulo $p$, let 
\[A_k(\G_0^n(p),\psi) =  \{ f \colon \ha_n \to \C  \mid  f|_k \g (Z) = \psi(\det D_\g) f(Z) \ \text{for $\g \in \G^n_0(p)$} \}, \]
and for $n \ge 2$
\[ M_k(\G_0^n(p), \psi) = \{ f \in A_k(\G^n_0(p), \psi) \mid \text{$f$ is holomorphic} \}, \]
that is the space of holomorphic Siegel modular forms of degree $n$, level $p$ and character $\psi$. If $n=1$ we also require the holomorphy on cusps. We only consider the case $\psi =\chi_0$ or $\chi_p$, here $\chi_0$ stands for the trivial character modulo $p$ and $\chi_p$ stands for the quadratic character modulo $p$. 

The set of symmetric matrix of size $n$ with entries in $R$ is denoted by $\Sym^n(R)$. We put $\Sym^n(\Z)^*$ the set of half integral matrices. For $A \in M_n(\C)$, we put $\epi(A) = \exp 2 \pi \sqrt{-1} \Tr(A)$.
\end{notation}

\section{The action of $U(p)$-operator on the space of Siegel Eisenstein series} \label{section_U(p)-operator}

In this section we study the action of  $U(p)$-operator on the space of Siegel Eisenstein series.  It is already investigated by Walling (\cite[Section 4]{Wal}), that works for any square-free level.
Here we introduce another method. We only consider the level $p$ case, however we can avoid complicated calculation by lifting modular forms of classical notations to modular forms of adelic notations.

\subsection{$U(p)$-operator in classical notation} \label{subsection_U(p)-classical}

Let $\tau_p = \begin{pmatrix} \1_n & 0 \\ 0 & p \1_n \end{pmatrix}$. 
For $f \in A_k(\G_0^n(p),\psi)$, we define  $U(p)$-operator on the space $A_k(\G^n_0(p),\psi)$ by
\[ \bigl(U(p)f \bigr)(Z)  = p^{nk/2-n(n+1)/2} \sum_{\g \in \G^n_0(p) \backslash \G^n_0(p) \tau_p \G^n_0(p)} \psi(\det A_\g) f|_k \g. \] 
It is well-known that a representative set of $\G^n_0(p) \backslash \G^n_0(p) \tau_p  \G^n_0(p)$ is given by
\[ \left\{ \begin{pmatrix} \1_n & S \\ 0 & p  \1_n \end{pmatrix} \biggm| S \in \Sym^n(\Z/p) \right\}, \]
thus we have
\[ (U(p)f)(Z) = p^{-n(n+1)/2} \sum_{S \in \Sym^n(\Z/p)} f \left( \frac{Z+S}{p} \right). \]
Because of the above, the $U(p)$-action becomes quite simple in terms of Fourier coefficients. 
For $f \in M_k(\G_0^n(p), \psi)$, if the Fourier expansion of $f$ is given by \[ f(Z) = \sum_{T \in \Sym^n(\Z)^*} C(T) \epi(TZ), \]
then
\[ \bigl(f|U(p) \bigr)(Z) = \sum_{T \in \Sym^n(\Z)^*} C(pT) \epi(TZ). \]

Next we define the space of Siegel Eisenstein series. For $0 \le r \le n$, we put 
\begin{equation} w_{n,r} = w_r = \left( \begin{array}{cc|cc} 0_{r} & & -\1_r & \\ & \1_{n-r} & & 0_{n-r} \\ \hline \1_r & & 0_r & \\  & 0_{n-r} & & \1_{n-r} \end{array} \right). \label{eq_def_w_j}
\end{equation}
We set
\begin{align*} P_n(\R) & = \left\{ \begin{pmatrix} \mu A & B \\ 0 & \tp \! A^{-1}  \end{pmatrix} \in \GSp(n,\R) \right\}, \\
P_{0,n}(\R) & = \left\{ \begin{pmatrix} A & B \\ 0 & \tp \! A^{-1}  \end{pmatrix} \in \Sp(n,\R) \right\}.
\end{align*}
Then a representative set of $P_{0,n}(\Z) \backslash \Sp(n,\Z)/\G^n_0(p)$ is given by $\{ w_r \}_{0 \le r \le n}$.

Assume that $\psi = \chi_0$ or $\chi_p$. We define the function $\psi^*$ on  $P_{0,n}(\Z) w_r \G^n_0(p)$ by 
\begin{equation} \psi^*(\g) = \psi(\det D_\eta) \psi(\det D_\kappa), \quad  \g = \eta w_r \kappa, \ \eta \in P_{0,n}(\Z), \ \kappa \in \G_0^n(p). \label{eq_psi*}
\end{equation}
One can check that $\psi^*$ is well-defined.
Let $k$ be an integer such that $\psi(-1) = (-1)^k$. For $Z = X+\sqrt{-1}Y \in \ha_n$, we define the Siegel Eisenstein series of degree $n$, weight $k$, level $p$ with character $\psi$ corresponding to the cusp $w_r$ by
\[ E^n_{k,\psi}(w_r;Z,s) = \det(Y)^{s-k/2}  \sum_{\g} \psi^*(\g) \,  j(\g, Z)^{-k} \, |j(\g,Z)|^{-2s+k},
\]
here $\g$ runs thorough the set $P_{0,n}(\Z) \backslash P_{0,n}(\Z) w_r \G_0^n(p)$.
The right hand side converges when $2\re(s) > n+1$. As is well-known, it continues meromorphically in whole $s$-plane (cf. \cite[Theorem 1]{Kal} and \cite[Proposition 2.1]{Shi2}). Note that if $\psi$ is neither trivial nor quadratic, then $E^n_{k,\psi}(w_r;Z,s)$ is defined only when $r = 0$ or $n$. We have $E^n_{k,\psi}(w_r;Z,s) \in A_k(\G_0^n(p), \psi)$ and $E^n_{k,\psi}(w_r;Z, k/2) \in M_k(\G_0^n(p), \psi)$.
If the the Fourier expansion is given by
\[ E^{n}_{k,\psi}(w_r;Z,s) = \sum_{T \in \Sym^n(\Z)^*} b(T,Y,s) \epi(TX), \]
then
\begin{equation} U(p) E^{n}_{k,\psi} (w_r; Z,s)  = \sum_{T \in \Sym^n(\Z)^*} b(pT, p^{-1}Y,s) \epi(TX). \label{eq_U(p)_with_parameter_s}
\end{equation}

We define the space of Siegel Eisenstein series by
\begin{align*}
\mathcal{E}_{k,s}(\G^n_0(p), \psi) & = \langle E^n_{k,\psi}(w_r;Z,s) \mid 0 \le r \le n \rangle_\C.
\end{align*}
This space is closed under the action of $U(p)$-operator.

\subsection{$U(p)$-operator in adelic notatoion} \label{subsection_U(p)-G(A)}

In order to investigate the action of $U(p)$-operator, it is convenient to consider it in adelic notation. Let $\A$ be the adele ring of $\Q$. Let $G = \GSp_n$ and $G_\infty^+ = \GSp^+(n,\R)$. 
We define the compact group $\prod_{v \le \infty} K_v \subset G(\mathbb{A})$ by
\[ K_v = \begin{cases} \Sp(n,\R) \cap \mathit{O}(2n) & \text{if $v = \infty$,} 
\\ \GSp(n,\Z_q) & \text{if $v = q \ne p$ is a finite place,} \\ 
\{ \g \in \GSp(n,\Z_p) \mid C_\g \equiv 0 \bmod p \} & \text{if $v=p$.} \end{cases}
\]
Though $\prod_v K_v$ is not a maximal compact subgroup of $G(\A)$,  we have a decomposition
\begin{equation} \label{eq_strong_approx_GSp}
G(\A) = G(\Q) G_\infty^+ \prod_{\text{$q$:prime}} K_q. 
\end{equation}
The Dirichlet character $\psi$ modulo $p$ can be lifted to the idele character $\w$ as follows. Let $\A^\times$ be the idele group of $\Q^\times$, then we have
\[ \quad  \A^{\times} = \Q^\times \R^\times_+ \prod_{\text{$q$:prime}} \Z_q^\times, \]
with the set $\R^\times_+$ of positive real numbers. For  $g = \g \delta \alpha \in \A^\times$ with $\g \in \Q^\times$, $\delta \in \R^\times_+$ and $\alpha = (\alpha_q)_q \in \prod \Z_q^\times$ as above decomposition, we define the character $\w \colon \Q^\times \backslash \A^\times \to \C^\times$ by
\[ \w(g) = \psi(\varpi(\alpha_p))^{-1} \]
under the natural surjection $\varpi \colon \Z_p^\times \to (\Z/p)^\times$. We denote by $\w_v$ the $v$-component of $\w$, i.e.\ $\w_v  \colon \Q_v^\times \hookrightarrow \A^\times \xrightarrow{\w} \C^\times$.  Written explicitly, if $q \ne p$ be a finite place then
\[ \w_q( q^m \alpha) = \psi(q)^m \quad \text{with $\alpha \in \Z_q^\times$}, \]
and
\[ \w_p(p^m \alpha) = \psi( \varpi(\alpha) )^{-1} \quad \text{with $\alpha \in \Z_p^\times$.} \]
For the infinite place $v=\infty$,
\[ \w_\infty(d) = \psi \left( \frac{d}{|d|} \right) = \begin{cases} 1 & \text{if $\psi(-1) = 1$} \\ \mathrm{sign}(d) & \text{if $\psi(-1) = -1$}. \end{cases} \]

For $f \in A_k(\G_0^n(p), \psi)$ we define the $\C$-valued function $\Lambda(f)$ on $G(\A)$ as follows. Let $\bfi_n = \sqrt{-1} \cdot \1_n \in \ha_n$. 
For $g = \g g_\infty \kappa \in G(\A)$  as in (\ref{eq_strong_approx_GSp}),  we define
\[ \Lambda(f)(g) = \w_p(\det D_{\kappa_p}) j(g_\infty, \bfi_n)^{-k} f(g_\infty \langle \bfi_n \rangle). \]
This is well-defined since $f \in A_k(\G^n_0(p),\psi)$. Then $\Phi=\Lambda(f)$ satisfies the following properties.
\begin{equation} \label{eq_adelic_modular_form}
\begin{split}
& \Phi(\g g) = \Phi(g)  \quad \text{for} \quad  \g \in G(\Q), \\
& \Phi(g u) =  \Phi(g) j(u, \bfi_n)^{-k} \quad \text{for} \quad  u \in K_\infty, \\
& \Phi(g \kappa) = \Phi(g) \w_p( \det D_{\kappa_p}) \quad  \text{for} \quad  \kappa =(\kappa_q) \in \prod_q K_q.
\end{split}
\end{equation}

Conversely for a function $\Phi \colon G(\A) \to \C$ satisfying (\ref{eq_adelic_modular_form}), we define the function $\Pi(\Phi) \colon \ha_n \to \C$ by 
\[ \Pi(\Phi)(Z) = \Phi(g_\infty) j(g_\infty, \bfi_n)^k, \]
here $g_\infty \in \GSp^+(n,\R)$ is taken so that $g_\infty \langle  \bfi_n \rangle = Z$. Then $\Pi(\Phi) \in A_k(\G_0^n(p), \psi)$, and the maps  $\Lambda$ and $\Pi$ are inverse of each other.

Let $\mathcal{S}(\A)$ be the space of the smooth functions on $G(\A)$, that is $C^\infty$-class in infinite place and locally constants in finite places. We consider the $U(p)$-operator on $\mathcal{S}(\A)$. Put
\[ H_p = K_p  \begin{pmatrix} p \1_n & 0 \\ 0 & \1_n \end{pmatrix} K_p  \]
and we define the function $\widetilde{\w}_p \colon H_p \to \C$ by $\widetilde{\w}_p(x) = \w_p(\det D_x)$. For $\Phi \in \mathcal{S}(\A)$ satisfying (\ref{eq_adelic_modular_form}) we define
\begin{equation}  \label{eq_def_U(p)_adele}
(U(p)\Phi)(g) = p^{nk/2-n(n+1)/2} \int_{H_p} \widetilde{\w}_p^{-1} (x) \Phi(gx) \, d x,
\end{equation}
here $d x$ is the Haar measure of $\GSp(\Q_p)$ such that the volume of $K_p$ equals to $1$. Then $U(p)(\Phi)$ also satisfies (\ref{eq_adelic_modular_form}).
Since we have
\begin{equation*} 
 H_p = \coprod_{u \in \Sym^n(\Z_p/p)} \begin{pmatrix} p \1_n & u \\ 0 & \1_n \end{pmatrix} K_p,
\end{equation*}
we can write
\begin{equation} \label{eq_def_U(p)_adele_another}
 (U(p)\Phi)(g) = p^{nk/2-n(n+1)/2} \sum_{u \in \Sym^n(\Z_p/p)} \Phi \left( g \begin{pmatrix} p \1_n & u \\ 0 & \1_n \end{pmatrix}_p \right).
\end{equation}

\begin{lemma}  \label{lemma_compatible_U(p)}
Let $f \in A_k(\G^n_0(p), \psi)$ be a smooth function. Then we have \[U(p)(\Lambda(f)) = \Lambda( U(p)f). \] 
\end{lemma}

\begin{proof} For simplicity we write $d(n,k) = nk/2 - n(n+1)/2$.
Take $g = \g g_\infty \kappa \in G(\A)$ with $\g \in G(\Q)$, $g_\infty \in \GSp^+(n,\R)$ and $\kappa = (\kappa_q)_q \in \prod_q K_q$. In (\ref{eq_def_U(p)_adele}) change the variable $x \mapsto \kappa_p^{-1}x$ then
\[ \bigl( U(p)\Lambda(f) \bigr)(g) = p^{d(n,k)} \w_p(\det D_{\kappa_p}) \int_{H_p} \widetilde{\w}_p^{-1}(x) \Lambda(f)(g_\infty \kappa' x) \,dx,
\]
with $\kappa' =(\kappa'_q)_q$ such that $\kappa_q' = \kappa_q$ for $q \ne p$ and $\kappa_p' = \1_{2n}$.  By (\ref{eq_def_U(p)_adele_another}) we have
\[(U(p) \Lambda(f))(g) = p^{d(n,k)} \w_p(\det D_{\kappa_p}) \sum_{u \in \Sym^n(\Z_p/p)} \Lambda(f) \left( g_\infty \kappa' \begin{pmatrix} p \1_n & u \\ 0 & \1_n \end{pmatrix}_p \right). \]
We put $\nu_u = \begin{pmatrix} p \1_n & u \\ 0 & \1_n \end{pmatrix} \in G(\Q)$. Then
\[ g_\infty \kappa' \begin{pmatrix} p \1_n & u \\ 0 & \1_n \end{pmatrix}_p = \nu_u \widetilde{g}_\infty \widetilde{\kappa} \]
with
\[ \widetilde{g}_\infty = 
  (\nu_u^{-1})_\infty g_\infty, \quad \widetilde{\kappa}_p = \1_n. 
\]
If we write $Z = g_\infty \langle  \bfi_n \rangle \in \ha_n$, then\begin{align*}
\Lambda(f)(\nu_u \widetilde{g}_\infty \widetilde{\kappa}) & = j\bigl( ( \nu_u^{-1})_\infty g_\infty,  \bfi_n \bigr)^{-k} f((\nu_u^{-1})_\infty g_\infty \langle  \bfi_n \rangle) \\
& = j( \nu_u^{-1}, Z )^{-k} j(g_\infty,  \bfi_n)^{-k} f \left( \nu_u^{-1} \langle Z \rangle \right),
\end{align*}
Thus
\begin{align*}& \quad  \bigl( U(p)\Lambda(f) \bigr)(g)    = \w_p(\det D_{\kappa_p}) j(g_\infty,  \bfi_n)^{-k} p^{d(n,k)}\sum_{u \in \Sym^n(\Z_p/p)} f|_k \nu_u^{-1}( Z). 
\end{align*}
Since $\nu_u^{-1} = p^{-1} \begin{pmatrix} \1_n & -u \\ 0 & p \1_n \end{pmatrix}$, we have proved our lemma.

\end{proof}

\subsection{Siegel Eisenstein series on $G(\A)$} \label{subsection_Siegel_Eisen_G(A)}

Now we consider the space of Eisenstein series. 
For each place $v$, let $\mathcal{S}_v$ be the set of smooth functions $f \colon G(\Q_v) \to \C$. We set for $s \in \C$
\[I_v^0(s,\w) = \left\{ f \in \mathcal{S}_v \biggm|  f \Bigl( \begin{smatrix} \mu \, \tp \! D^{-1} & * \\ 0 & D \end{smatrix} g \Bigr) = \w_v (\det D) \bigl| \mu^n \det D^{-2} \bigr|_v^{s/2} f(g) \right\}.  \]
We also define
\[I_v^1(s,\w) = \left\{f \in \mathcal{S}_v \biggm|  f \Bigl( \begin{smatrix} \mu \, \tp \! D^{-1} & * \\ 0 & D \end{smatrix} g \Bigr) = \w_v(\mu \det D) \, \bigl| \mu^n \det D^{-2} \bigr|_v^{s/2} f(g) \right\}.  \]
If the idele character $\w$ is  trivial  then $I_v^0(s,\w) = I_v^1(s,\w)$. 

Fix an integer $k$ such that $\psi(-1) = (-1)^k$. For $\nu = 0$ or $1$, we define the space $I_v^{\nu}(s, \w)_{K_v}$ as follows.
If $v = q \ne p$ is a finite place then
\[ I_q^\nu(s, \w)_{K_q} = \{ f \in I_q^\nu(s, \w) \mid f(g \kappa) = f(g) \ \text{for $\kappa \in K_q$} \}. \]
If $v = p$, then
\[ I_p^\nu(s, \w)_{K_p} = \{ f \in I_p^\nu(s, \w) \mid f(g \kappa) = \w_p(\det D_\kappa) f(g)  \ \text{for $\kappa \in K_p$}  \}. \]
Finally if $v = \infty$ then
\[ I^\nu_\infty(s, \w)_{K_\infty} = \{f \in I^\nu_\infty(s, \w) \mid f(g \kappa ) = j(\kappa, \bfi_n)^{-k} f(g) \ \text{for $\kappa \in K_\infty$}  \}. \]

For $v \ne p$, the $\C$-vector space $I^\nu_v(s, \w)_{K_v}$ is 1-dimensional, since in those cases we have $G(\Q_v) = P_n(\Q_v) K_v$. 
We take the basis $f^{(\nu)}_v$ of $I_v^\nu(s, \w)_{K_v}$ so that $f^{(\nu)}_v(\1_{2n}) = 1$.
On the other hand for $v=p$, we have the decomposition
\[ G(\Q_p) = \coprod_{r=0}^n P_n(\Q_p) w_r K_p. \]

\begin{lemma} Assume that $\w$ is the quadratic character. There exists $h \in I_p^\nu(s,\w)_{K_p}$ such that $h(w_i) \ne 0$ if and only if $i \equiv \nu \bmod 2$.

\end{lemma}

\begin{proof} Suppose that $\eta w_i = w_i \kappa$ with $\eta \in P_n(\Q_p)$ and $\kappa \in K_p$ holds, then all the entries of $\eta$ lie in $\Z_p$. Put $\mu = \mu(\eta) = \mu(\kappa)$. There exisits $h \in I_p^\nu(s,\w)_{K_p}$ such that $h(w_i) \ne 0$, if and only if $\w_p(\mu)^\nu \w_p(\det D_\eta) \, \bigl| \mu^n \det D_\eta^{-2} \bigr|_p^{s/2} = \w(\det D_\kappa)$ holds.
We write
\[ \eta = \left( \begin{array}{cc|cc} a_1 & a_2 & b_1 & b_2 \\ a_3 & a_4 & b_3 & b_4 \\ \hline 0 & 0& d_1 & d_2 \\ 0 & 0 & d_3 & d_4 \end{array} \right), \quad a_1, b_1, d_1 \in M_i(\Z_p). \]
Then
\[ \kappa = \left( \begin{array}{cc|cc} d_1 & 0 & 0 & d_2 \\ b_3 & a_4 & -a_3 & b_4 \\ \hline -b_1 &  -a_2 & a_1 & -b_2 \\ d_3 & 0 & 0 & d_4 \end{array} \right), \quad \begin{pmatrix} -b_1 & -a_2 \\ d_3 & 0 \end{pmatrix} \equiv 0_n \bmod p. \]	
We see that  all of $\det a_1$, $\det a_4$, $\det d_1$ and $\det d_4$ are contained in $\Z_p^\times$ and  $\w_p (\det D_\eta) = \w_p(\det d_1) \w_p( \det d_4)$. 
On the other hand, since $A_\eta \tp \!D_\eta = \mu \1_n$, we have $a_1 \tp d_1 \equiv \mu \1_i \bmod p$, thus  $\w_p(\det a_1 \det d_1) = \w_p(\mu)^i$. Hence
\[ \w_p(\mu)^i \w_p (\det D_\eta) \, \bigl|\mu^n \det D_\eta^{-2} \bigr|^{s/2}_p = \w_p(\det a_1 \det d_4 ) = \w_p(\det D_\kappa), \]
thus the desired condition is $\w_p(\mu)^i = \w_p(\mu)^\nu$. This proves our lemma.
\end{proof}

We put
\[ I_p (s,\w)_{K_p} = \begin{cases} I_p^0(s,\w)_{K_p} \vrule width0pt depth7pt & \text{$\w$ is trivial,} \\  I_p^0(s,\w)_{K_p} \oplus I_p^1(s,\w)_{K_p} & \text{$\w$ is quadratic.} \end{cases} \] 
By the above lemma, we have $\dim I_p(s, \w)_{K_p} = n+1$. We take the basis $\{ h_i \}_{0 \le i \le n}$ of $I_p(s, \w)_{K_p}$ so that the support of $h_i$ is  $P_n(\Q_p) w_i K_p$ and $h_i(w_i) = 1$. Note that if $\w$ is quadratic, then $h_{2i} \in I^0_p(s,\w)_{K_p}$ and $h_{2i-1} \in I^1_p(s,\w)_{K_p}$.

For $\varphi \in I_p^\nu(s, \w)_{K_p}$, we define the function $\bff_\varphi^{(\nu)} \colon G(\A) \to \C$ by
\[ \bff_\varphi^{(\nu)}( (g_v)_v ) = \varphi(g_p) \prod_{v \ne p} f^{(\nu)}_v(g_v), \quad (g_v)_v \in G(\A), \]
then $\bff_\varphi^{(\nu)}$ is left $P_n(\Q)$ invariant. Take $\varphi \in I_p(s,\w)_{K_p}$, if $\w$ is quadratic we write $\varphi= \varphi_0 + \varphi_1$ with $\varphi_\nu \in I^{\nu}_p(\w,s)_{K_p}$. We set 
\[ \bff_\varphi = \begin{cases} \bff^{(0)}_\varphi & \text{$\w$ is trivial,} \\ \bff^{(0)}_{\varphi_0} + \bff^{(1)}_{\varphi_1} & \text{$\w$ is quadratic.}  \end{cases} \]
For $g \in G(\A)$ and $\varphi \in I_p(s,\w)_{K_p}$ the function
\begin{equation} \label{eq_def_Eisenstein_adelic}
E( \bff_\varphi, g, s) = \sum_{\g \in P_n(\Q) \backslash G(\Q)} \bff_\varphi(\g g)
\end{equation}
converges when $\re(s) \gg 0$ and $E(\bff_\varphi,g,s)$ satisfies (\ref{eq_adelic_modular_form}). 

\begin{lemma} Let $\{ h_j \}$ be the basis of $I_{p}(2s,\w)_{K_p}$ defined as above. For $0 \le j \le n$ we have
\[ \Pi \bigl( E(\bff_{h_j}, g, 2s) \bigr)(Z) = E^n_{k,\psi}(w_j;Z,s). \]
\end{lemma}

\begin{proof} Take $g_\infty \in G^+_\infty$ such that $g_\infty \langle \bfi_n \rangle = Z$. Since we have $P_n(\Q) \Sp(n,\Z) = G(\Q)$, one can write
\[ E(\bff_{h_j}, g_\infty, 2s)  = \sum_{\g \in P_{0,n}(\Z) \backslash \Sp(n,\Z)} \bff_{h_j}(\g g_\infty). \]
Put $\nu = 0$ or $1$ according as $j$ is even or odd, if $\w$ is quadratic. Then $\bff_{h_j}(\g g_\infty) = h_j(\g) f_\infty^{(\nu)}(\g g_\infty)$, thus we may assume $\g \in P_{0,n}(\Z) w_j \G_0^n(p)$ and we have
\[ \Pi \bigl( E(\bff_{h_j}, g, 2s) \bigr)(Z) = \sum_{ \g \in P_{0,n}(\Z) \backslash P_{0,n}(\Z) w_j \G^n_0(p)} \psi^*(\g) f_\infty^{(\nu)}(\g g_\infty) j(g_\infty,  \bfi_n)^k. \]
We write $\g g_\infty = \eta \kappa$ with $\eta \in P_{0,n}(\R)$ and $\kappa \in K_\infty$. Then
\begin{align*} f_\infty^{(\nu)} (\g g_\infty) & = \w_\infty(\det D_\eta) | \det D_\eta|^{-2s} j(\kappa,  \bfi_n)^{-k} 
\\
& =  \w_\infty (\det D_\eta)   |\det D_\eta|^{-2s} \det(D_\eta)^{k} j(\g,Z)^{-k} j(g_\infty, \bfi_n)^{-k}.
\end{align*}
We see that $\w_\infty(\det D_\eta) \det(D_\eta)^{k} = |\det D_\eta|^{k}$ by $\psi(-1) = (-1)^k$. Moreover
$\g \langle Z \rangle = \eta \langle  \bfi_n \rangle = B_\eta D_\eta^{-1} + \sqrt{-1} \, \tp \!D_\eta^{-1} D_\eta^{-1}$ shows 
\[ |\det D_\eta|^{-2s+k} = |\det ( \im \g \langle Z \rangle)|^{s-k/2} = \det (Y)^{s-k/2} \, |j(\g,Z)|^{-2s+k}. \]
This proves our lemma.
\end{proof}

By the above lemma  the map
\begin{equation} \label{isom_degenerate_principal_eisen} I_p(2s, \w)_{K_p} \to \mathcal{E}_{k,s}(\G^n_0(p), \psi), \quad \phi \mapsto \Pi\bigl( E( \bff_\phi, g, s) \bigr) 
\end{equation}
is an isomorphism of $\C$-vectors. We can consider the $U(p)$-operator defined in (\ref{eq_def_U(p)_adele}) on  $I(2s, \w)_{K_p}$, then the isomorphism (\ref{isom_degenerate_principal_eisen}) is $U(p)$-equivariant thanks to Lemma \ref{lemma_compatible_U(p)}. 
Therefore in order to study the $U(p)$-operator on the space of Siegel Eisenstein series, it suffices to consider the $U(p)$-operator (\ref{eq_def_U(p)_adele_another}) on $I_p(2s, \w)_{K_p}$. 

\subsection{$U(p)$-operator on the space of Eisenstein series} \label{subsection_U(p)-action_Eisen}

Now we investigate the action of the operator $U(p)$ on $I_p(2s,\w)_{K_p}$. Let $\varphi \in I_p(2s, \w)_{K_p}$. We put $d(n,k) = nk/2-n(n+1)/2$ as before. By (\ref{eq_def_U(p)_adele_another}) we have
\[ (U(p) \varphi)(w_j) = p^{d(n,k)} \sum_{u \in \Sym^n(\Z_p/p)} \varphi \left( w_j \begin{pmatrix} p \1_n & u \\ 0 & \1_n \end{pmatrix} \right). \]
Here we omit the subscript $p$ representing the $p$-component, for we only consider the $p$-local part. 
Write $u = \begin{pmatrix} u_1 & u_2 \\ \tp u_2 & u_4 \end{pmatrix}$ with $u_1 \in \Sym^j(\Z_p/p)$, $u_2 \in M_{j,n-j}(\Z_p/p)$ and $u_4 \in \Sym^{n-j}(\Z_p/p)$.
Then
\begin{align*} w_j  \begin{pmatrix} p \1_n & u \\ 0 & \1_n \end{pmatrix} & = \left( \begin{array}{cc|cc} 0 & 0 & -\1_j & 0 \\ 0 & p\1_{n-j} & \tp u_2 & u_4 \\ \hline p \1_j &0 & u_1 & u_2 \\ 0 & 0 & 0 & \1_{n-j} \end{array} \right) \\ 
& = \left( \begin{array}{cc|cc} \1_j & 0 & 0 & 0 \\ -\tp u_2  & \1_{n-j} & 0 & u_4 \\ \hline 0 &0 & \1_j & u_2 \\ 0 & 0 & 0 & \1_{n-j} \end{array} \right)
\left( \begin{array}{cc|cc} 0 & 0 & -\1_j & 0 \\ 0 & p\1_{n-j} & 0 & 0 \\ \hline p \1_j &0 & u_1 & 0 \\ 0 & 0 & 0 & \1_{n-j} \end{array} \right).
\end{align*}
We denote the latter matrix of the second line by $M(u_1)$, then $\varphi \left( w_j \begin{pmatrix} p \1_n & u \\ 0 & \1_n \end{pmatrix} \right) = \varphi(M(u_1))$. Since $d(n,k) + j(n-j) + (n-j)(n-j+1)/2 = nk/2-j(j+1)/2$, we have 
\[ (U(p)\varphi)(w_j)  = p^{nk/2-j(j+1)/2}\sum_{u_1 \in \Sym^j(\Z_p/p)} \varphi(M(u_1)). \]
For $u_1 \in \Sym^j(\Z_p)$ we can take $\g \in \GL_j(\Z_p)$ such that
\[ u_1 = \tp \g \begin{pmatrix} p \nu_i & 0 \\ 0 & \la_{j-i} \end{pmatrix} \g \]
with $\la_{j-i} \in \GL_{j-i}(\Z_p) \cap \Sym^{j-i}(\Z_p)$. Set $T(\g) = \diag(\tp \g^{-1}, \1_{n-i}, \g, \1_{n-i})$. Then we have
$\tp T(\g)^{-1} M(u_1) T(\g)^{-1} = M( \begin{smatrix} p \nu_i \\ & \la_{j-i} \end{smatrix})$, thus
\[ \varphi(M(u_1)) = \w(\g)^2 \varphi \bigl( M\begin{smatrix} p \nu_i \\ & \la_{j-i} \end{smatrix} \bigr) = \varphi \bigl( M\begin{smatrix} p \nu_i \\ & \la_{j-i} \end{smatrix} \bigr). \]
We have
\begin{align*} & M \bigl (\begin{smatrix} p \nu_i \\ & \la_{j-i} \end{smatrix} \bigr)  =  \left( \begin{array}{cc|cc} 0_j & 0 & -\1_j & 0 \\ 0 & p\1_{n-j} & 0 & 0_{n-j} \\ \hline p \1_j &0 & \begin{smallmatrix} p \nu_i \\ & \la_{j-i} \end{smallmatrix} \vrule width0pt height13pt depth5pt & 0 \\ 0 & 0_{n-j} & 0 & \1_{n-j} \end{array} \right)  \\
& = \left( \begin{array}{cc|cc} \1_i & & 0_i & \\ & p \1_{n-i} & & \begin{smallmatrix} - \la_{j-i}^{-1} \\ & 0_{n-j} \end{smallmatrix} \vrule width0pt depth9pt \\ \hline  & & p \1_i \\ & & &  \1_{n-i} \end{array} \right) w_i \left( \begin{array}{cc|cc} \1_i &  & \nu_i &  \\ &  \begin{smallmatrix} \la_{j-i}^{-1} \\ & \1_{n-i} \end{smallmatrix}  &  & 0_{n-i} \\ \hline  0_i & &   \1_i &   \\ & \begin{smallmatrix} p\1_{j-i} & \\ & 0_{n-j} \end{smallmatrix} & & \begin{smallmatrix}  \la_{j-i} \\ &  \1_{n-j} \end{smallmatrix}  \end{array} \right),
\end{align*}
the final matrix of the second line is contained in $K_p$. Hence we have the following result. For $u_1 \in \Sym^j(\Z_p)$ with $u_1 = \tp \g \begin{pmatrix} p \nu_j & 0 \\ 0 & \la_{j-i} \end{pmatrix} \g$, we have
\[ \varphi(M(u_1)) = \w(\det \la_{j-i}) p^{-ns+2si} \varphi(w_i). \]

Recall that $\psi = \chi_0$ or $\chi_p$ is the Dirichlet characters modulo $p$.  For positive integers $l$ and $m$ with $l \ge m$, we set
\[ W^l_m(\psi) = \sum_{\substack{ u \in \Sym^l(\Z/p) \\ \rank(u \bmod p) = m}} \widehat{\psi}(u), \]
here $\widehat{\psi}(u) = \psi(\det \la)$ with $u = \tp \g \begin{pmatrix} \la & \\ &  0_{l-m} \end{pmatrix} \g$ for some $\g \in \GL_l(\Z/p)$. We also set $W^m_0(\psi) = 1$ for $m \ge 0$.
Using it we get the following proposition.

\begin{proposition} Let $\{ h_i \}$ $(0 \le i \le n)$ be the basis of $I_p(2s, \w)_{K_p}$ such that the support of $h_i$ is $P_n(\Q_p)w_iK_p$ and $h_i(w_i) = 1$. Then
\[ (U(p)h_i)(w_j) =  \begin{cases} p^{n(k/2-s) + 2si - j(j+1)/2} W_{j-i}^j(\psi) & j \ge i \\ 0 & j<i . \end{cases} \]
\end{proposition}

In particular the representative matrix of the $U(p)$-operator on the space of Siegel Eisenstein series with respect to the basis $\{ E^n_{k,\psi}(w_i; Z,s) \}$ is a triangular matrix. Thus we have the following property.

\begin{corollary}[cf.\ {\cite[Section 4, Proposition]{Boe}, \cite[Corollary 4.2]{Wal}}]  \label{cor_eigenvalues_U(p)} Let
\[ l(s,j) = l_k(s,j) = n(k/2-s) + 2sj - \frac{j(j+1)}{2}. \]
The eigenvalues of $U(p)$-operator on $\mathcal{E}_{k,s}(\G^n_0(p),\psi)$ are given by
$ \{ p^{l(s,j)} \mid  0 \le j \le n \}$. 
In particular, substituting $s=k/2$, the eigenvalues of $U(p)$-operator on the space of holomorphic Siegel Eisenstein series of weight $k$, degree $n$, level $p$ with character $\psi$ are given by $p^{kj-j(j+1)/2}$ $(0 \le j \le n)$.
\end{corollary}

Finally we can compute the value $W^j_{j-i}(\psi)$  as follows.

\begin{lemma} Assume that $l \ge m$.
\begin{enumerate}[$(1)$]
\item For the trivial character $\chi_0$ modulo $p$ we have
\[ W^l_{m}(\chi_0) = \begin{cases} p^{m(m+2)/4} \ \dfrac{ \prod_{r=l-m+1}^l (p^r-1)}{\prod^{m/2}_{r=1} (p^{2r}-1)} & \text{$m$ is even} \\ \\
p^{ (m^2-1)/4} \ \dfrac{ \prod_{r=l-m+1}^l (p^r-1)}{\prod^{(m-1)/2}_{r=1} (p^{2r}-1)} & \text{$m$ is odd.}  
\end{cases} \]
\item For the quadratic character $\chi_p$ modulo $p$ we have
\[ W^l_m(\chi_p) = \begin{cases} \chi_p(-1)^{m/2} p^{m^2/4} \ \dfrac{ \prod_{r=l-m+1}^l (p^r-1)}{\prod^{m/2}_{r=1} (p^{2r}-1)} & \text{$m$ is even} \\ 0 & \text{$m$ is odd.}\end{cases} \]
\end{enumerate}

\end{lemma}

\begin{proof}  For $A \in \Sym^r(\F_p)$ we put $\itO_r(A) = \{ \g \in \GL_r(\F_p) \mid \tp \g A \g = A \}$. For $T(\la) =\diag(\la,0_{l-m}) \in \Sym^l(\F_p)$ with $\la \in \Sym^m(\F_p) \cap \GL_m(\F_p)$, we have
\[ \itO_l( T(\la) ) = \left\{ \begin{pmatrix} a & 0 \\ b & c \end{pmatrix} \biggm| a \in \itO_m(\la), \ b \in M_{l-m,l}(\F_p), c \in \GL_{l-m}(\F_p) \right\}. \]
Fix an element $\delta \in \F_p^\times$ such that $\chi_p(\delta) = -1$ and put $E_m = \diag(1,\ldots, 1, \delta) \in \Sym^m(\F_p)$. It is well-known that for any $u \in \Sym^m(\F_p)$ with $\det u \ne 0$, there exists $\g \in \GL_m(\F_p)$ such that $\tp \g u \g = \1_m$ or $E_m$. Thus we have
\begin{align*} W^l_m(\chi_0) & = \sharp \bigl( \GL_l(\F_p)/ \itO_l(T(\1_m)) \bigr) + \sharp \bigl( \GL_l(\F_p)/ \itO_l(T(E_m)) \bigr), \\ 
W^l_m(\chi_p)  &= \sharp \bigl( \GL_l(\F_p)/\itO_l(T(\1_m)) \bigr)-\sharp \bigl( \GL_l(\F_p) /\itO_l(T(E_m)) \bigr).
\end{align*}
Hence the assertion follows from the fact
\[ \sharp \GL_l(\F_p) = p^{l(l-1)/2} \prod_{r=1}^l (p^r-1) \]
and the proposition below.
\end{proof}

\begin{proposition} 
\begin{enumerate}[$(1)$]
\item If $m$ is odd then
\[ \sharp \itO_m(\1_m) = \sharp \itO_m(E_m) = 2 p^{(m-1)^2/4} \prod_{r=1}^{(m-1)/2} (p^{2r}-1). \]
\item If $m$ is even then
\begin{align*}
\sharp \itO(\1_m) & = 2 p^{(m^2-2m)/4} \bigl( p^{m/2}-\chi_p(-1)^{m/2} \bigr) \prod_{r=1}^{m/2-1} (p^{2r}-1 ), \\
\sharp \itO(E_m) & = 2 p^{(m^2-2m)/4} \bigl( p^{m/2}+\chi_p(-1)^{m/2} \bigr) \prod_{r=1}^{m/2-1} (p^{2r}-1 ).
\end{align*}
\end{enumerate}
\end{proposition}

As a consequence we have the following result. In order to state the concrete description we prepare the notation. For $0 \le i,j \le n$ we put
\begin{equation*}
m_{\chi_0}(s)_{ij} = 
 \begin{cases} p^{2si -j(j+1)/2 + (j-i)(j-i+2)/4} \\ \qquad \quad \times 
\dfrac{\prod_{r=i+1}^j(p^r-1)}{\prod_{r=1}^{(j-i)/2} (p^{2r}-1)} & \text{if $i \le j$ and $j-i$ is even,}  \\ \\
p^{2si -j(j+1)/2 + ((j-i)^2-1)/4} \\ 
\qquad \quad \times  \dfrac{\prod_{r=i+1}^j(p^r-1)}{\prod_{r=1}^{(j-i-1)/2} (p^{2r}-1)} & \text{if $i \le j$ and $j-i$ is odd,} \\ 
0 & \text{if $i>j$.} \end{cases}
\end{equation*}
We also set
\begin{equation*}
m_{\chi_p}(s)_{ij} = 
 \begin{cases} \chi_p(-1)^{(j-i)/2} \, p^{2si -j(j+1)/2 + (j-i)^2/4} \\  \qquad  \times  \dfrac{\prod_{r=i+1}^j(p^r-1)}{\prod_{r=1}^{(j-i)/2} (p^{2r}-1)} & \text{if $i \le j$ and $j-i$ is even,}  \\ \\
0 & \text{otherwise.} \end{cases}
\end{equation*}

\begin{theorem}[cf.\ {\cite[Theorem 4.1]{Wal}}]
\begin{enumerate}[$(1)$]
\item For the trivial character case we have
\[ U(p) E^n_{k,\chi_0} (w_i;Z,s)  = p^{n(k/2-s)} \sum_{j=i}^n m_{\chi_0}(s)_{ij} \, E^n_{k,\chi_0}(w_j;Z,s). \]

\item For the quadratic character case we have
\[ U(p) E^n_{k,\chi_p} (w_i;Z,s)  = p^{n(k/2-s)} \sum_{\substack{j=i \\ j \equiv i \bmod 2}}^n m_{\chi_p}(s)_{ij} \, E^n_{k,\chi_p}(w_j;Z,s). \]

\end{enumerate}
\end{theorem}

For $\psi = \chi_0$ or $\chi_p$, let $M_\psi(s)$ be the $(n+1) \times (n+1)$-matrix whose $(i,j)$-component is $m_\psi(s)_{ij}$, here $i$ and $j$ run through  $0 \le i,j \le n$. Then the representative matrix of $U(p)$ as the endomorphism of $\mathcal{E}_{k,s}(\G^n_0(p), \psi)$ with respect to the basis $ \{ E^n_{k,\psi}(w_i; Z, k) \mid 0 \le i \le n \}$ is given by $p^{n(k/2-s)} \, \tp \! M_\psi(s)$, that is a lower triangular matrix. 

We use the easy lemma for linear algebra.
\begin{lemma} Let $A = (a_{ij})_{0 \le i \le n, \, 0 \le j \le n}$ be a lower triangular matrix. Assume that the diagonal entries $\la_i = a_{ii}$ are mutually distinct. We define $v_j^{(i)}$ inductively by
\[ v_j^{(i)} = \begin{cases} 0 & 0 \le j \le i-1 \\ 1 & j = i \\ \displaystyle{- \frac{1}{\la_j - \la_i} \sum_{k=i}^{j-1} a_{jk} v_k^{(i)} } & j \ge i+1. \end{cases} \]
Then the column vector $\mathbf{v}^{(i)} = (v_j^{(i)})_{0 \le j \le n}$ is an eigenvector of $A$ with respect to the eigenvalue $\la_i$.
\end{lemma}

For $\psi = \chi_0$ or $\chi_p$, we define $b_\psi(s)_{ij}$ with $0 \le i,j \le n$ inductively as follows: 
\begin{equation} \label{eq_b_psi}
\begin{split}
& \text{if $i>j$} \quad b_\psi(s)_{ij}=0,   \\ 
& \text{if $i=j$} \quad b_\psi(s)_{ii}=1, \\
& \text{if $i<j$} \quad b_\psi(s)_{ij} = -\frac{p^{-2si + i(i+1)/2}}{p^{(j-i) (2s-(j+i+1)/2 )} -1} \sum_{r=i}^{j-1} m_\psi(s)_{rj} b_\psi(s)_{ir}. 
\end{split} 
\end{equation} Note that $b_{\chi_p}(s)_{ij}=0$ unless $i \not\equiv j \bmod 2$.
Let $B_\psi(s)$ be the $(n+1)\times (n+1)$-matrix  whose $(i,j)$-component is $b_\psi(s)_{ij}$. 

Then by the above lemma we have the following proposition.

\begin{proposition} \label{prop_eigen_Eisenstein_series}
For $\psi = \chi_0$ or $\chi_p$, the Eisenstein series
\[ E^{n,(i)}_{k,\psi}(Z,s) = E^n_{k,\psi}(w_i;Z,s) + \sum_{j=i+1}^n b_\psi(s)_{ij} E^n_{k,\psi}(w_j;Z, s) \]
is an eigenfunction of $U(p)$, whose eigenvalue is $p^{n(k/2-s) + 2si-i(i+1)/2}$.

\end{proposition}

\section{Fourier expansion of the Siegel Eisenstein series} \label{section_Fourier_exp}

In order to write down the functional equations explicitly, we consider the Fourier expansion of the Siegel Eisenstein series. 
We put for $n \ge r$,
\[ \Z_{\rm prim}^{(n,r)} = \{A \in M_{n,r}(\Z) \mid \exists B \in M_{n,n-r}(\Z) \ \text{such that} (A,B) \in \GL_n(\Z) \}. \]
For matrices $A$ and $B$ we set $A[B] = \tp \! B A B$ if the matrix products are defined.
Then by a similar argument of \cite[p.306]{Ma} we have the following.
\begin{proposition} \label{prop_Fourier_exp_w_r}
\begin{align*}
E^n_{k,\psi}(w_\nu;Z,s) & = (\det Y)^{s-k/2} \sum_{r=\nu}^n \sum_{N \in \Sym^r(\Z)^*}  \mathscr{S}_{r}^{\nu}(\psi,N,2s) \\ 
& \times  \sum_{Q \in \Z^{(n,r)}_{\mathrm{prim}}/\GL_r(\Z)}  \xi_{r} \Bigl(Y[Q], N, s+\frac{k}{2}, s-\frac{k}{2} \Bigr) \epi(N[ \tp \! Q]X).
\end{align*}
\end{proposition}

Here $\xi_r(g,h,\alpha,\beta)$ is the confluent hypergeometric function investigated in \cite{Shi1}; $\mathscr{S}_r^{\nu}(\psi,N,s)$ is a kind of Siegel series defined as follows. Let
\[ \mathcal{M}_r = \left\{(C,D) \in M_{r,2r}(\Z) \Bigm|  \begin{array}{l} \text{$(C,D)$ is symmetric and co-prime,} \\ \det C > 0 \end{array} \right\}. \]
Here $(C,D)$ is symmetric if $C \, \tp \! D = D \, \tp  C$ and $(C, D)$ is co-prime if there exists $U,V \in M_r(\Z)$ such that $CU+DV = \1_r$. Then the map 
\[ \SL_r(\Z) \backslash \mathcal{M}_r \to \Sym^r(\Q),  \quad (C,D) \mapsto C^{-1}D \] 
is bijective. For $T = C^{-1}D \in \Sym^r(\Q)$ we put $\delta(T) = \det C$.
For $0 \le \nu \le r$ set
\begin{align*} \mathcal{M}^\nu_r & = \{(C,D) \in \mathcal{M}_r \mid \rank(C \bmod p) = \nu \}, \\
 \Sym^r(\Q)^{(\nu)} & = \{ R = C^{-1} D \in \Sym^{r}(\Q) \mid (C,D) \in \mathcal{M}^\nu_r \}.
\end{align*}
We define the function $\widetilde{\psi}$ on $\Sym^r(\Q)^{(\nu)}$ as follows. For $R = C^{-1} D\in \Sym^r(\Q)^{(\nu)}$ with $(C,D) \in \mathcal{M}_r^{\nu}$, there exists $\g = \begin{smatrix} * & * \\ C & D \end{smatrix} \in P_{0,r}(\Z) w_{r,\nu} \G^r_0(p)$.
Then we put $\widetilde{\psi}(R) = \psi^*(\g)$, where $\psi^*$ is defined in (\ref{eq_psi*}). Under the above notation we define
\[ \mathscr{S}^{\nu}_r(\psi,N,s) = \sum_{R \in \Sym^{r}(\Q/\Z)^{(\nu)} } \widetilde{\psi}(R) \delta(R)^{-s} \epi(RN), \]
here we write $\Sym^r(\Q/\Z)^{(\nu)}$ instead of $\Sym^r(\Q)^{(\nu)} \bmod \Sym^r(\Z)$.

We can show that the Siegel series $\mathscr{S}_r^\nu(\psi,N,s)$ has an Euler product expression. 
 For a prime number $q$, let $\Sym^r(\Q)_q = \bigcup_{i=0}^\infty \dfrac{1}{q^i} \Sym^r(\Z)$. Any $R \in \Sym^r(\Q)$ has a decomposition $R = R_0 + \sum_{i=1}^s R_i$ with $R_0 \in \Sym^r(\Q)_p$ and $R_i \in \Sym^r(\Q)_{q_i}$, where  $p, q_1, \ldots, q_s$ are mutually distinct primes. This decomposition is unique up to $\Sym^r(\Z)$. Then we have 
\[ \delta(R) = \delta(R_0) \prod_{i=1}^s \delta(R_i). \]
We put $\Sym^r(\Q)^{(\nu)}_p = \Sym^r(\Q)^{(\nu)} \cap \Sym^r(\Q)_p$. 

\begin{lemma} \label{lem_euler_product_S_w_r} Suppose that $R \in \Sym^r(\Q)$ is written as  $R = R_0 + \sum_{i=1}^s R_i$ with $R_0 \in \Sym^r(\Q)_p$ and $R_i \in \Sym^r(\Q)_{q_i}$. 
Then $R \in \Sym^r(\Q)^{(\nu)}$ if and only if $R_0 \in \Sym^r(\Q)^{(\nu)}_p$,
and in that case we have
\[ \widetilde{\psi}(R) = \widetilde{\psi}(R_0) \prod_{i=1}^r \psi(\delta(R_i)). \]
\end{lemma}

\begin{proof}
For $R \in \Sym^r(\Q)$, there exists $U,V \in \SL_r(\Z)$ such that
\[ R = U\begin{pmatrix} \la_1/\delta_1 \\ & \ddots & \\ && \la_r/\delta_r \end{pmatrix}V, \quad \delta_i > 0, \ \delta_i \mid \delta_{i+1}, \  (\delta_i, \la_i)=1, \]
by the elementary divisor theorem. Put
\[ C = \begin{pmatrix} \delta_1 \\ & \ddots \\ && \delta_r \end{pmatrix} U^{-1}, \quad D = \begin{pmatrix} \la_1 \\ & \ddots \\ && \la_r \end{pmatrix} V, \]
then $(C,D) \in \mathcal{M}_r$ and $R = C^{-1}D$.
Assume that $R \in \Sym^r(\Q)^{(\nu)}$ for some $\nu$, then $(p, \delta_i) = 1$ for $1 \le i \le \nu$ and $p \mid \delta_i$ for $\nu+1 \le i \le r$.  Since $(CU, D \, \tp  U^{-1})$ is also symmetric and co-prime, if we put $W = V  \, \tp  U^{-1}$ and write
\[ W = \begin{pmatrix} W_1 & W_2 \\ W_3 & W_4 \end{pmatrix} \quad W_1 \in M_\nu(\Z), \ W_4 \in M_{r-\nu}(\Z), \]
then $W_3 \equiv 0 \bmod p$. In this case we can show 
\[ \widetilde{\psi}(R) = \psi(\det W_4) \prod_{i=1}^\nu \psi(\delta_i) \prod_{i=\nu+1}^r \psi(\la_i).\] 

Next we consider $R_0$. Put $\delta = \delta(R) = \prod_{i=1}^r \delta_i$. Let $e = \ord_p \delta$ and  $\delta = p^e \delta'$. If we write $1/\delta = m/p^e + l/\delta'$, then we can take $R_0 = \delta'm R$. We also put $e_i = \ord_p \delta_i$ for $\nu+1 \le i \le r$ and write $\delta_i = p^{e_i} \delta'_i$. Then $R_0 = C_0^{-1} D_0$ with
$C_0 = \diag(1, \ldots, 1, p^{e_{\nu+1}}, \ldots, p^{e_r})U^{-1}$ and
\[ D_0  = \begin{pmatrix} * \\ & m \delta' \la_{\nu+1}/\delta'_{\nu+1}  \\ && \ddots \\ &&& m \delta' \la_{r}/\delta_{r}' \end{pmatrix} V. \]
Thus we have proved that $R \in \Sym^r(\Q)^{(\nu)}$ if and only if $R_0 \in \Sym^r(\Q)^{(\nu)}_p$. Moreover, since $m \delta' \equiv 1 \bmod p^e$ by definition, we have
\[ \widetilde{\psi}(R_0) = \psi(\det W_4) \prod_{i=\nu+1}^r \psi(\la_i) \prod_{i=\nu+1}^r \psi(\delta'_i)^{-1}. \]
Since $\delta' = \bigl( \prod_{i=1}^\nu \delta_i \bigr) \bigl( \prod_{i=\nu+1}^r \delta'_i \bigr)$, we have $\widetilde{\psi}(R) = \widetilde{\psi}(R_0) \psi(\delta')$, which proves the lemma.
\end{proof}

We define the local Siegel series. For a prime number $q$, put
\[ \mathcal{M}_r(q) = \left\{(C,D) \in M_{r,2r}(\Z_q) \Bigm|  \begin{array}{l} \text{$(C,D)$ is symmetric and co-prime,} \\ \text{$\det C = q^l$ for some $l \ge 0$}  \end{array} \right\}. \]
Then the map $\SL_r(\Z_q) \backslash \mathcal{M}_r(q) \to \Sym^r(\Q_q)$, $(C,D) \mapsto C^{-1} D$ is bijective. For $T = C^{-1} D \in \Sym^r(\Q_q)$ we put $\delta_q(R) = \det C$. We also set for $0 \le \nu \le r$
\begin{align*} \mathcal{M}^\nu_r(p) & = \{(C,D) \in \mathcal{M}_r(p) \mid \rank(C \bmod p) = \nu \}, \\
 \Sym^r(\Q_p)^{(\nu)} & = \{ R = C^{-1} D \in \Sym^{r}(\Q_p) \mid (C,D) \in \mathcal{M}^\nu_r(p) \}.
\end{align*}
The function $\widetilde{\psi}_p$ on $\Sym^r(\Q_p)^{(\nu)}$ is defined similarly to $\widetilde{\psi}$. For $T \in M_r(\Z_q)$, put $\epi_q(T) = \exp  2 \pi \sqrt{-1} \varpi \bigl( \Tr(T) \bigr)$, with natural bijection $\varpi: \Q_q/\Z_q \simeq \bigcup_{i \ge 0} \dfrac{1}{q^i} \Z/\Z$.
The local Siegel series are defined as 
\begin{align*}
S_r(\psi,N,s)_q & = \sum_{R \in \Sym^r(\Q_q/\Z_q) } \psi(\delta_q(R)) \delta_q(R)^{-s}  \epi_q(RN), \quad (q \ne p) \\ 
 S_r^{\nu}(\psi, N,s)_p & = \sum_{R \in \Sym^r(\Q_q/\Z_q)^{(\nu)}} \widetilde{\psi}_p(R) \delta_p(R)^{-s}  \epi_p(RN).
\end{align*}

It is known that these series are rational functions in $q^{-s}$; thus by abuse of language we also write them $S_r(\psi,N,q^{-s})_q$ or $S_r^{\nu}(\psi,N,p^{-s})_p$. The ordinary Siegel series are given by
\[ S_r(N,q^{-s})_q = \sum_{R \in \Sym^r(\Q_q/\Z_q)} \delta_q(R)^{-s} \epi_q(RN), \]
then we have $S_r(\psi,N,q^{-s})_q = S_r(N,\psi(q)q^{-s})_q$ for $q \ne p$.

By Lemma \ref{lem_euler_product_S_w_r}, we have the follwing.

\begin{lemma} The Siegel series $\mathcal{S}_r^{\nu}(\psi,N,s)$ has an Euler product expression
\[ \mathscr{S}_r^{\nu}(\psi, N, s ) = S_{r}^{\nu}(\psi, N,s)_p \, \prod_{q \ne p} S_r (\psi,N,s)_q. \]
\end{lemma}

In order to give an explicit form of the functional equations, it suffices to compute the constant term of the Fourier expansion. For the Eisenstein series $E^n_{k,\psi}(w_\nu; Z,s)$, the constant term is given by
\[ \det(Y)^{s-k/2} \mathscr{S}_r^\nu(\psi,0,2s) \sum_{r=\nu}^n \sum_{Q \in \Z^{(n,r)}_{\rm prim}/ \GL_r(\Z)} \xi_r \Bigl( Y[Q], 0, s+\frac{k}{2}, s-\frac{k}{2} \Bigr). \]
 By \cite[(4.9), (4,34K)]{Shi1},  we can write it $\sum_{r=\nu}^n \mathscr{S}_r^\nu(\psi,0,2s) C_r(Y,k,s)$ with
\begin{align} C_r(Y,k,s) &  =  (\sqrt{-1})^{-r k} 2^{r(r+3)/2-2r s} \pi^{r(r+1)/2} \dfrac{\Gamma_r \left( 2s-\dfrac{r+1}{2} \right)}{\Gamma_r \left(s+\dfrac{k}{2} \right) \Gamma_r\left(s-\dfrac{k}{2}\right)} \notag \\ 
& \qquad \times (\det Y)^{s-k/2} \zeta_r^{(n)}\left(Y,2s-\dfrac{r+1}{2} \right). \label{eq_C(Y,k,s)}
\end{align}
Here $\Gamma_r(s) = \pi^{r(r-1)/4} \prod_{i=0}^{r-1} \Gamma(s-i/2)$ and $\zeta^{(n)}_r(Y,s)$ is the kind of Epstein zeta function given by
\[ \zeta^{(n)}_r(Y,s) = \sum_{Q \in \Z^{(n,r)}_{\rm prim}/\GL_r(\Z)} \det(Y[Q])^{-s}. \]
Note that $\zeta^{(n)}_n(Y,s) = (\det Y)^{-s}$ and we define $\zeta^{(n)}_0(Y,s) = 1$. The analytic property of $\zeta^{(n)}_r(Y,s)$ is studied in \cite{Miz}.

It is easy to see that
\[C_r(p^{-1}Y,k,s) = p^{l(s,r)} C_r(Y,k,s),\]
here $l(s,r)$ is defined in Corollary \ref{cor_eigenvalues_U(p)}. By (\ref{eq_U(p)_with_parameter_s}) and the construction of the $U(p)$-eigenfunction $E^{n,(\nu)}_{k,\psi}(Z,s)$ (Proposition \ref{prop_eigen_Eisenstein_series}), we have the following.

\begin{lemma} \label{lemm_const_term_eigen_function}
The constant term of the Fourier expansion of $E^{n,(\nu)}_{k,\psi}(Z,s)$ is given by $\mathscr{S}^{\nu}_\nu(\psi,0,2s) C_\nu(Y,k,s)$.
\end{lemma}

For the Siegel series, clearly we have $S_\nu^{\nu}(\psi,0,2s)_p = 1$. On the other hand for $q \ne p$, the following result is known.

\begin{lemma}[{\cite[Proposition 5.1]{Shi2}, \cite[Corollary 2]{Ki}}] \label{lem_unramified_S(psi,0)} For $q \ne p$ we have
\[ S_\nu(\psi,0,2s)_q = \frac{1-\psi(q)q^{-2s}}{1-\psi(q)q^{\nu-2s}} \prod_{i=1}^{[\nu/2]} \frac{1-q^{2i-4s}}{1-q^{2\nu+1 -2i-4s}}. \]
\end{lemma}

Thus we can write
\[ \mathscr{S}_\nu^{\nu}(\psi,0,2s) = \frac{L(2s-\nu,\psi)}{L(2s,\psi)} \prod_{i=1}^{[\nu/2]} \frac{L(4s+2i-2\nu-1,\chi_0)}{L(4s-2i,\chi_0)}. \]

\begin{remark} As far as the author knows, an explicit formula of $S_r^{\nu}(\psi,0,2s)_p$ is not yet known in general. However by looking at the constant term of the Fourier expansion we have
\[ E_{k,\psi}^n(w_\nu;Z,s) = \sum_{r=\nu}^n S_r^{\nu}(\psi,0,2s)_p  \, E_{k,\psi}^{n,(r)}(Z,s), \]
since for $r \ge \nu$ we have
\[ S_r^{\nu}(\psi,0,2s)_p \cdot \mathscr{S}_r^{r}(\psi,0,2s) = \mathscr{S}_r^{\nu}(\psi,0,2s). \]
Thus $S_r^{\nu}(\psi,0,2s)_p$ is the $(\nu,r)$-component of the matrix $B_\psi(s)^{-1}$, here $B_\psi(s)$ is the strict upper triangular matrix defined in (\ref{eq_b_psi}).
\end{remark}

\section{Functional equations} \label{section_functional_eq}

Now we consider the functional equations of Siegel Eisenstein series of level $p$. 

\subsection{The case of the trivial character}
First we consider the case of $\psi = \chi_0$. 
Let $\kappa_n = (n+1)/2$.
First we remark the following fact.

\begin{lemma} \label{lem_FE_rough_form}
Let $E^{n,(\nu)}_{k, \chi_0}(Z,s)$ be the $U(p)$-eigenfunction defined as above. Then there exists a functional equation   
\[ E^{n,(\nu)}_{k, \chi_0}(Z,s) = C(s) E^{n, (n-\nu)}_{k,\chi_0}(Z, \kappa_n-s), \]
here $C(s)$ is a function in $s$.
\end{lemma}

\begin{proof}
Let
\[ E^{n}_k (Z,s) = \det(Y)^{s-k/2} \sum_{\g \in P_{n,0}(\Z) \backslash \Sp(n,\Z)} j(\g,Z)^{-k} \, |j(\g,Z)|^{-2s+k} \]
be the Siegel Eisenstein series of level $1$. By \cite[Section 4, Proposition]{Boe}, we have
\[ \left\langle U(p^i) E^{n}_k (Z,s )  \bigm| 0 \le i \le n \right\rangle_\C =   \mathcal{E}_{k,s}(\G^n_0(p), \chi_0). \] 
By Corollary \ref{cor_eigenvalues_U(p)} the eigenvalues of the $U(p)$-operator are given by $p^{l(s,j)}$ $(0 \le j \le n)$ with $l(s,j) = n(k/2-s)+2sj-j(j+1)/2$. Thus we have
\[ E^{n, (\nu)}_{k,\chi_0} (Z,s ) = A(s) \Bigl( \prod_{j \ne \nu}(U(p)-p^{l(s,j)} ) \Bigr)  E^{n}_k (Z,s )   \]
for some function $A(s)$, that is independent of $Z$.

By \cite[Corollary 6.6]{Miz}, the functional equation of the form
\[ E^{n}_k (Z,s ) = D(s) E^{n}_k (Z,\kappa_n-s ) \]
holds. Thus
\begin{align*} E^{n,(\nu)}_{k,\chi_0} (Z,s ) & = A(s) \Bigl(\prod_{j \ne \nu}(U(p)-p^{l(s,j)}) \Bigr)  E^{n}_k (Z,s)  \\
& = A(s) D(s) \Bigl( \prod_{j \ne \nu}(U(p)-p^{l(s,j)}) \Bigr)  E^{n}_k (Z,\kappa_n-s).
\end{align*}
Since  $l(s,j) = l\bigl(\kappa_n-s,n-j \bigr)$, we have
\[  E^{n,(\nu)}_{k,\chi_0} (Z,s) =A(s)D(s)A(\kappa_n-s )^{-1}  E^{n,(n-\nu)}_{k,\chi_0} (Z,\kappa_n-s),\]
that proves the lemma.
\end{proof}

In order to write down the functional equations explicitly, it suffices to see the constant terms of the Fourier expansions. 
Let
\[ \xi(s) = \pi^{-s/2} \Gamma\left( \frac{s}{2} \right) \zeta(s) =\xi(1-s) \]
be the completed Riemann zeta function.
Though we use the same letter as the confluent hypergeometric function $\xi_n(g,s,\alpha, \beta)$ appeared in the previous section,  it will not occur confusion.
Following \cite{Miz} we define
\begin{align*} 
& \mathbb{E}_{k,\chi_0}^{n,(\nu)}(Z,s)  = \g_\nu(\chi_0,s) \dfrac{\Gamma_n\left(s + \dfrac{k}{2} \right)}{\Gamma_n(s)} \, \xi(2s) \prod_{j=1}^{[n/2]} \xi(4s-2j) \,  E_{k,\chi_0}^{n,(\nu)}(Z,s),
\end{align*}
with
\[ \g_\nu(\chi_0,s) =  \frac{1-p^{-2s}}{1-p^{\nu-2s}} \prod_{i=1}^{[\nu/2]} \frac{1-p^{2i-4s}}{1-p^{2\nu+1 -2i-4s}}.
\]
Then by (\ref{eq_C(Y,k,s)}), Lemma \ref{lemm_const_term_eigen_function} and Lemma \ref{lem_unramified_S(psi,0)}, the constant term $\Lambda_{k,\chi_0}^{n,(\nu)}(Y,s)$ of the Fourier expansion of $\mathbb{E}^{n,(\nu)}_{k,\chi_0}(Z,s)$ is given by 
\begin{align*} \Lambda_{k,\chi_0}^{n,(\nu)}(Y,s) & =  (\sqrt{-1})^{-\nu k} 2^{\nu(\nu+3)/2-2\nu s} \pi^{\nu(\nu+1)/2} \dfrac{\Gamma_\nu \left( 2s-\dfrac{\nu+1}{2} \right)\Gamma_n\left(s + \dfrac{k}{2} \right)}{\Gamma_\nu \left(s+\dfrac{k}{2} \right) \Gamma_\nu\left(s-\dfrac{k}{2}\right)\Gamma_n(s)}  \\
& \quad \times \dfrac{\zeta(2s-\nu)}{\zeta(2s)} \prod_{i=1}^{[\nu/2]} \dfrac{\zeta (4s+2i-2\nu-1)}{\zeta (4s-2i)} \xi(2s) \prod_{j=1}^{[n/2]} \xi(4s-2j) \\
& \quad \times (\det Y)^{s-k/2} \zeta_\nu^{(n)}\left(Y,2s-\dfrac{\nu+1}{2} \right).
\end{align*}

\begin{proposition} \label{prop_FE_constant_term_chi_0}
We have
\[ \Lambda_{k,\chi_0}^{n,(\nu)}(Y,s) = \Lambda_{k,\chi_0}^{n,(n-\nu)}(Y,\kappa_n-s). \]
\end{proposition}

Together with Lemma \ref{lem_FE_rough_form} and Proposition \ref{prop_FE_constant_term_chi_0}, our main result follows.
\begin{theorem} \label{thm_FE_eigen_function_chi_0}
We have a functional equation
\[ \mathbb{E}^{n,(\nu)}_{k,\chi_0}(Z,s) = \mathbb{E}^{n,(n-\nu)}_{k,\chi_0}(Z, \kappa_n-s). \]
\end{theorem}

We can also write down the functional equations in terms of the basis $\{ E^{n}_{k,\chi_0}(w_i;Z,s) \}_{0 \le i \le n}$ of the space $\mathcal{E}^n_{k,s}(\G^n_0(p),\chi_0)$. We define 
\[\mathbb{E}_{k,\chi_0}^{n}(w_\nu;Z,s)  = \dfrac{\Gamma_n\left(s + \dfrac{k}{2} \right)}{\Gamma_n(s)} \, \xi(2s) \prod_{j=1}^{[n/2]} \xi(4s-2j) \,  E_{k,\chi_0}^n(w_\nu;Z,s). \]

\begin{theorem} \label{thm_FE_chi_0_standard_basis}
Let $T_{\chi_0}(s)$ be the anti-diagonal matrix of size $n+1$, whose \\ $(n-\nu,\nu)$-component is $\g_\nu(\chi_0,s) \g_{n-\nu}(\chi_0, \kappa_n-s)^{-1}$ for $0 \le \nu \le n$. Then we have
\begin{align*}  \begin{pmatrix} \mathbb{E}^{n}_{k,\chi_0} (w_0;Z,\kappa_n-s ) \\ \vdots \\ \mathbb{E}^{n}_{k,\chi_0}(w_n;Z,\kappa_n-s ) \end{pmatrix} = B_{\chi_0}(\kappa_n-s)^{-1} T_{\chi_0}(s) B_{\chi_0}(s) \begin{pmatrix} \mathbb{E}^{n}_{k,\chi_0}(w_0;Z,s) \\ \vdots \\ \mathbb{E}^{n}_{k,\chi_0}(w_n;Z,s) \end{pmatrix},  
\end{align*}
here $B_{\chi_0}(s)$ is the upper triangular matrix defined in (\ref{eq_b_psi}).

\end{theorem}

In order to prove the Proposition \ref{prop_FE_constant_term_chi_0} we quote two lemmas from \cite{Miz}.

\begin{lemma}[{\cite[Lemma 6.1]{Miz}}] \label{lem_FE_Gamma_function}  For $r \in \Z$ we have
\[ \frac{\Gamma_m(s)}{\Gamma_m(s+r)} = (-1)^{mr} \frac{ \Gamma_m(\kappa_m-r-s)}{\Gamma_m(\kappa_m-s)}. \]

\end{lemma}

\begin{lemma}[{\cite[Lemma 6.2]{Miz}}] \label{lem_FE_epstein_zeta} For $0 \le \nu \le n$, we have
\[ \zeta^{(n)}_{n-\nu}(Y,s) = \det(Y)^{\nu/2-s} \, \frac{\prod_{j=n-\nu}^{n-1} \, \xi(2s-j)}{\prod_{j=0}^{\nu-1} \, \xi(2s-j)} \zeta_{\nu}^{(n)}\left( Y, \frac{n}{2}-s \right). \] 

\end{lemma}

\begin{proof}[Proof of Proposition \ref{prop_FE_constant_term_chi_0}] We put $\la= n-\nu$. 
Multiplying suitable power of $\pi$ and certain Gamma functions in order to change zeta functions into completed zeta functions, we have
\begin{equation}  \Lambda^{n,(\nu)}_{k,\chi_0}(s)  = (\sqrt{-1})^{-\nu k} 2^{\nu(\nu+3)/2-2 \nu s} \pi^{\nu(\nu+1)/4} G_\nu(s) Z_\nu(s) \label{eq_Lambda_Z(s)_G(s)}
\end{equation}
with
\[ G_\nu(s) = \dfrac{\Gamma_\nu \left( 2s-\dfrac{\nu+1}{2} \right)\Gamma_n\left(s + \dfrac{k}{2} \right)}{\Gamma_\nu \left(s+\dfrac{k}{2} \right) \Gamma_\nu\left(s-\dfrac{k}{2}\right)\Gamma_n(s)} \cdot \dfrac{\Gamma(s)}{\Gamma\left(s-\dfrac{\nu}{2} \right)}  \prod_{i=1}^{[\nu/2]} \frac{\Gamma(2s-i)}{\Gamma \left(2s+i- \dfrac{2\nu+1}{2} \right)} \]
and
\begin{align*} Z_\nu(s) & = \xi(2s-\nu) \prod_{j=1}^{[\nu/2]}\xi(4s+2j-2\nu-1) \prod_{j=[\nu/2]+1}^{[n/2]} \xi(4s-2j) \\
& \quad \times  (\det Y)^{s-k/2} \zeta^{(n)}_\nu \left(Y, 2s- \frac{\nu+1}{2} \right).
\end{align*}

First we treat $Z_\nu(s)$. We have
\[ \prod_{j=1}^{[\nu/2]}\xi(4s+2j-2\nu-1) = \prod_{\substack{j=2 \\ \text{even}}}^{\nu} \xi(4s+j-2\nu-1) = \prod_{\substack{j=0 \\ j \equiv \nu \, (2)}}^{\nu-1} \xi(4s-\nu-1-j), \]
the latter is obtained by the translation $j \mapsto \nu-j$. Further
\[ \prod_{j=[\nu/2]+1}^{[n/2]} \xi(4s-2j) =  \prod_{\substack{j=\nu+1 \\ \text{even}}}^n \xi(4s-j) = \prod_{\substack{j=0 \\ j \not\equiv \nu \, (2)}}^{\la-1} \xi(4s-\nu-1-j), \]
thus we can rewrite $Z_\nu(s)$ as follows. If $\nu \le \la$ then
\begin{align}  Z_\nu(s) & =  \xi(2s-\nu) \prod_{j=0}^{\nu-1} \xi(4s-\nu-1-j) \prod_{\substack{j=\nu \\ j \not\equiv \nu \, (2)}}^{\la-1} \xi(4s-\nu-1-j) \notag \\
& \quad \times (\det Y)^{s-k/2} \zeta^{(n)}_\nu \left(Y, 2s- \frac{\nu+1}{2} \right), \label{eq_Z_nu(s)_nu_le_la}
\end{align}
and if $\nu \ge \la$ then
\begin{align}  Z_\nu(s) & =  \xi(2s-\nu) \prod_{j=0}^{\la-1} \xi(4s-\nu-1-j) \prod_{\substack{j=\la \\ j \equiv \nu \, (2)}}^{\nu-1} \xi(4s-\nu-1-j) \notag \\& \quad \times (\det Y)^{s-k/2} \zeta^{(n)}_\nu \left(Y, 2s- \frac{\nu+1}{2} \right). \label{eq_Z_nu(s)_nu_ge_la}
\end{align}
By Lemma \ref{lem_FE_epstein_zeta} we have
\begin{align*} & (\det Y)^{s-k/2} \zeta^{(n)}_\nu \left( Y, 2s-\frac{\nu+1}{2} \right) \\
& = \frac{\prod_{j=\nu}^{n-1} \xi(4s-\nu-1-j)}{\prod_{j=0}^{\la-1}\xi(4s-\nu-1-j)} (\det Y)^{\kappa_n-s-k/2} \zeta_{\la}^{(n)} \left(Y, 2(\kappa_n-s)- \frac{\la+1}{2} \right).
\end{align*}
In the case  of $\nu \le \la$,  decompose $\prod_{j=0}^{\la-1} = \prod_{j=0}^{\nu-1} \cdot \prod_{j=\nu}^{\la-1}$ then we have
\begin{align*}
Z_\nu(s) & = \xi(2s-\nu) \prod_{j=\la}^{n-1} \xi(4s-\nu-1-j) \prod_{\substack{j=\nu \\ j \not\equiv \nu \, (2)}}^{\la-1} \xi(4s-\nu-1-j) \\
& \quad \times (\det Y)^{\kappa_n-s-k/2} \zeta_{\la}^{(n)} \left(Y, 2(\kappa_n-s)- \frac{\la+1}{2} \right).
\end{align*}
We have $\xi(2s-\nu) = \xi \bigl( 2(\kappa_n-s) - \la \bigr)$ and 
\begin{align*}
\prod_{j=\la}^{n-1} \xi(4s-\nu-1-j) & = \prod_{j=\la}^{n-1} \xi(2-4s + \nu + j) \\
& = \prod_{j=\la}^{n-1} \xi \bigl( 4(\kappa_n-s) - \la  -n+ j \bigr)  \\
& = \prod_{j=0}^{\nu-1} \xi \bigl( 4(\kappa_n-s)-\la-1-j \bigr),
\end{align*}
the final equality is obtained by $j \mapsto n-1-j$. Similarly we have
\[ \prod_{\substack{j=\nu \\ j \not\equiv \nu \, (2)}}^{\la-1} \xi(4s-\nu-1-j) = \prod_{\substack{j=\nu \\ j \equiv \la \, (2)}}^{\la-1} \xi \bigl( 4(\kappa_n-s)-\la-1-j), \]
therefore by comparing (\ref{eq_Z_nu(s)_nu_ge_la}) we have $Z_\nu(s) = Z_\la(\kappa_n-s)$. Also in the case of $\nu \ge \la$, one can show that  $Z_\nu(s) = Z_\la(\kappa_n-s)$ holds similarly.

Next we consider $G_\nu(s)$. For an integer $r$, we denote $\widetilde{r} = r(r-1)/4$ for simplicity.

If $\nu$ is even then by $i \mapsto \nu/2-i$ we have
\[ \prod_{i=1}^{[\nu/2]}\Gamma \left( 2s+i-\frac{2 \nu+1}{2} \right) = \prod_{i=0}^{\nu/2-1} \Gamma \left( 2s-\frac{\nu+1}{2}-i \right), \]
thus 
\begin{align*} & \quad  \  \Gamma_\nu \left(2s-\frac{\nu+1}{2} \right) \prod_{i=1}^{[\nu/2]} \frac{\Gamma(2s-i)}{\Gamma \left(2s+i -\dfrac{2 \nu+1}{2} \right)} \\
& =  \pi^{\widetilde{\nu}} \prod_{\substack{i=0 \\ \text{odd}}}^{\nu-1} \Gamma \left(2s-\frac{\nu+1}{2}-\frac{i}{2} \right) \prod_{i=1}^{\nu/2} \Gamma(2s-i)  = \pi^{\widetilde{\nu}} \prod_{i=1}^\nu \Gamma(2s-i). 
\end{align*}
By a similar calculation the above identity holds also in the case of $\nu$ is odd.
thus we have
\[ G_\nu(s) = \pi^{\widetilde{\nu}}  \frac{\displaystyle \Gamma_n \left( s + \dfrac{k}{2} \right) \, \prod_{i=1}^{\nu} \Gamma(2s-i)}{\Gamma_\nu \left( s + \dfrac{k}{2} \right)  \Gamma_\nu\left(s-\dfrac{k}{2} \right) \Gamma_n(s)} \cdot \frac{\Gamma(s)}{\Gamma \left(s- \dfrac{\nu}{2} \right)}.  \]
By Legendre's duplication formula we have
\[ \prod_{i=1}^{\nu} \Gamma(2s-i) = \pi^{-\nu/2} 2^{2s\nu-\nu(\nu+3)/2} \prod_{i=1}^\nu \left\{ \Gamma\Bigl( s- \frac{i}{2} \Bigr) \Gamma \Bigl(s-\frac{i}{2}+\frac{1}{2} \Bigr) \right\}, \]
therefore
\[ G_\nu(s) = \pi^{-\nu(\nu+1)/4} 2^{2s \nu - \nu(\nu+3)/2} \frac{ \Gamma_n\left( s+ \dfrac{k}{2} \right)  \Bigl( \Gamma_\nu(s) \Bigr)^2}{ \Gamma_\nu \left(s+\dfrac{k}{2} \right)\Gamma_\nu \left(s-\dfrac{k}{2} \right)\Gamma_n(s)}. \]
We have
\begin{align*} 
& \frac{\Gamma_n\left(s + \dfrac{k}{2} \right)}{\Gamma_\nu \left(s + \dfrac{k}{2} \right)}  = \pi^{\widetilde{n}-\widetilde{\nu}} \prod_{i=\nu}^{n-1} \Gamma \left( s + \frac{k}{2}-\frac{i}{2} \right) = \pi^{\widetilde{n}-\widetilde{\nu}-\widetilde{\la}} \, \Gamma_\la \left(s -\frac{\nu}{2} + \frac{k}{2}\right), \\
& \frac{\Gamma_\nu(s)}{\Gamma_n(s)}  = \pi^{-\widetilde{n}+ \widetilde{\nu}+\widetilde{\la}} \frac{1}{ \Gamma_\la \left(s- \dfrac{\nu}{2} \right)}. 
\end{align*}
As a consequence we have the following. 
Let 
\[G^*_\nu(s) = (\sqrt{-1})^{-\nu k} 2^{\nu(\nu+3)/2-2\nu s}\pi^{\nu(\nu+1)/4} G_\nu(s),\]
 then by (\ref{eq_Lambda_Z(s)_G(s)}) we have $\Lambda_{k,\chi_0}^{n,(\nu)}(s) = G^*_\nu(s) Z_\nu(s)$ and
\[ G^*_\nu(s) = (-1)^{-\nu k/2} \frac{ \Gamma_\nu(s)}{\Gamma_\nu \Bigl( s- \dfrac{k}{2} \Bigr)} \,  \frac{\Gamma_\la \Bigl( s - \dfrac{\nu}{2} +\dfrac{k}{2} \Bigr)}{ \Gamma_\la \Bigl(s- \dfrac{\nu}{2} \Bigr)}.  \]
We remark that $k$ is even since $\chi_0(-1) = 1$. By Lemma \ref{lem_FE_Gamma_function} we have
\begin{align*} \frac{ \Gamma_\nu(s)}{\Gamma_\nu \Bigl( s- \dfrac{k}{2} \Bigr)}   & = (-1)^{\nu k/2} \frac{\Gamma_\nu\Bigl(\kappa_n-s - \dfrac{\la}{2} + \dfrac{k}{2} \Bigr)}{\Gamma_\nu\Bigl( \kappa_n-s- \dfrac{\la}{2} \Bigr)}, \\
\frac{\Gamma_\la \Bigl( s - \dfrac{\nu}{2} +\dfrac{k}{2} \Bigr)}{ \Gamma_\la \Bigl(s- \dfrac{\nu}{2} \Bigr)} & = (-1)^{\la k/2} \frac{\Gamma_\la ( \kappa_n-s) }{\Gamma_\la \Bigl(\kappa_n-s - \dfrac{k}{2} \Bigr)},
\end{align*}
which shows that $G^*_\nu(s) = G^*_\la(\kappa_n-s)$. This completes the proof.
\end{proof}

\subsection{The case of the quadratic character}

Finally we consider the case of $\psi = \chi_p$. In this case we cannot induce the functional equations of $E^{n}_{k,\chi_p}(w_\nu;Z,s)$ from Theorem \ref{thm_FE_Siegel_Eisenstein}. 
However we also have the following. 
\begin{lemma} \label{lem_existence_FE_chi_p} There exists a functional equation of the form \[E^{n,(\nu)}_{k,\chi_p}(Z,s) = C(s) E^{n,(n-\nu)}_{k,\chi_p}(Z,\kappa_n-s).\]

\end{lemma}

\begin{proof} The desired assertion follows from the general theorem due to Langlands (\cite{La}). Using the notation in subsection \ref{subsection_Siegel_Eisen_G(A)}, put 
\[ I^{(\nu)}(s,\w)_K = \bigotimes_v I^{(\nu)}_v(s,\w)_{K_v} \]
for $\nu=0$ or $1$, and $I(s,\w)_K = I^{(0)}(s,\w)_K \oplus I^{(1)}(s,\w)_K$. We define the intertwining operator $M(s) \colon I(s,\w)_K \to I(n+1-s,\w)_K$ by
\[   M(s)f(g) = \int_{\Sym^n(\A)} f \left( w_n \begin{pmatrix} 1 & X \\ 0 & 1 \end{pmatrix} g \right) \, dX. \]
For $f \in I(s,\w)_K$, the Eisenstein series $E(f,g,s)$ is defined by
\[ E(f,g,s) = \sum_{\g \in P_n(\Q) \backslash G(\Q)} f(\g g). \]
Then Langlands showed that the  functional equation 
\[ E(f, g,s) = E(M(s)f, g, \kappa_n-s) \]
holds (cf.\ \cite[Part I, Main Theorem (a), (ii)]{Ar}, \cite[VI.2.1. Theorem (1)]{MW}).  We denote by $\varphi(\nu,2s)$ the $U(p)$-eigenfunction in $I_p(2s,\w)_{K_p}$ with the eigenvalue $p^{l(s,\nu)}$.  Since the intertwining operator $M(s)$ is compatible with $U(p)$-operator given in (\ref{eq_def_U(p)_adele_another}), we have $M(2s) \mathbf{f}_{\varphi(\nu,2s)} = C(s) \mathbf{f}_{\varphi(n-\nu, n+1-2s)}$ for a certain function $C(s)$, for $l(s,\nu) = l(\kappa_n-s, n-\nu)$. This proves our assertion.
 \end{proof}

We define the completed Dirichlet $L$-function as follows. For a Dirichlet character $\psi$ we put $\delta_\psi = (1 - \psi(-1))/2$. For simplicity we denote $\delta_p = \delta_{\chi_p}$, then $\delta_p=0$ or $1$ according as $p \equiv 1 $ or $p \equiv 3 \bmod 4$, equivalently $k$ is even or odd, for $\chi_p(-1)=(-1)^k$. The completed Dirichlet $L$-function is now defined by
\[ \xi(\chi_p,s) = \left( \frac{p}{\pi} \right)^{(s+ \delta_p)/2} \Gamma \left( \frac{s+ \delta_p}{2} \right) L(\chi_p,s), \]
that satisfies $\xi(\chi_p,s) = \xi(\chi_p,1-s)$.
We put $\ve_p = 1$ or $\sqrt{-1}$ according as $p \equiv 1$ or $p \equiv 3 \bmod 4$. For
\[ \g_\nu(\chi_p,s) =  (\ve_p p^{-1/2})^\nu  \prod_{i=1}^{[\nu/2]} \frac{1-p^{2i-4s}}{1-p^{2\nu+1 -2i-4s}},
\]
we define
\begin{align*} 
& \mathbb{E}_{k,\chi_p}^{n,(\nu)}(Z,s)  = \g_\nu(\chi_p,s) \dfrac{\Gamma_n\left(s + \dfrac{k}{2} \right)}{\Gamma_n \left( s + \dfrac{\delta_p}{2} \right)} \, \xi(\chi_p,2s) \prod_{j=1}^{[n/2]} \xi(4s-2j) \,  E_{k,\chi_0}^{n,(\nu)}(Z,s).
\end{align*}
Then the constant term of the Fourier expansion of $\mathbb{E}_{k,\chi_p}^{n,(\nu)}(Z,s)$ is given by
\begin{align*}  & \Lambda_{k,\chi_p}^{n,(\nu)}(Y,s)  =  (\sqrt{-1})^{-\nu k} 2^{\nu(\nu+3)/2-2\nu s} \pi^{\nu(\nu+1)/2}  \\
&  \quad \times (\ve_p p^{-1/2})^\nu  \dfrac{\Gamma_\nu \left( 2s-\dfrac{\nu+1}{2} \right)\Gamma_n\left(s + \dfrac{k}{2} \right)}{\Gamma_\nu \left(s+\dfrac{k}{2} \right) \Gamma_\nu\left(s-\dfrac{k}{2}\right)\Gamma_n\left( s + \dfrac{a}{2} \right)}  \\
&  \quad \times \dfrac{L(\chi_p, 2s-\nu)}{L(\chi_p, 2s)} \prod_{i=1}^{[\nu/2]} \dfrac{\zeta (4s+2i-2\nu-1)}{\zeta (4s-2i)} \xi(\chi_p, 2s) \prod_{j=1}^{[n/2]} \xi(4s-2j) \\
& \quad \times (\det Y)^{s-k/2} \zeta_\nu^{(n)}\left(Y,2s-\dfrac{\nu+1}{2} \right).
\end{align*}

The following similar results as the case of trivial character holds.
%

\begin{theorem} \label{thm_FE_eigen_function_chi_p}
We have a functional equation
\[ \mathbb{E}^{n,(\nu)}_{k,\chi_p}(Z,s) = \mathbb{E}^{n,(n-\nu)}_{k,\chi_p}(Z, \kappa_n-s). \]
\end{theorem}

 We put
\[\mathbb{E}_{k,\chi_p}^{n,}(w_\nu;Z,s)  = \dfrac{\Gamma_n\left(s + \dfrac{k}{2} \right)}{\Gamma_n \left(s + \dfrac{\delta_p}{2} \right)} \, \xi(\chi_p,2s) \prod_{j=1}^{[n/2]} \xi(4s-2j) \,  E_{k,\chi_p}^n(w_\nu;Z,s). \]

\begin{theorem}  \label{thm_FE_chi_p_standard_basis}
Let $T_{\chi_p}(s)$ be the anti-diagonal matrix of size $n+1$, whose \\ $(n-\nu,\nu)$-component is $\g_\nu(\chi_p,s) \g_{n-\nu}(\chi_p, \kappa_n-s)^{-1}$ for $0 \le \nu \le n$. We have
\begin{align*}  \begin{pmatrix} \mathbb{E}^{n}_{k,\chi_p} (w_0;Z,\kappa_n-s ) \\ \vdots \\ \mathbb{E}^{n}_{k,\chi_p}(w_n;Z,\kappa_n-s ) \end{pmatrix} = B_{\chi_p}(\kappa_n-s)^{-1} T_{\chi_p}(s) B_{\chi_p}(s) \begin{pmatrix} \mathbb{E}^{n}_{k,\chi_p}(w_0;Z,s) \\ \vdots \\ \mathbb{E}^{n}_{k,\chi_p}(w_n;Z,s) \end{pmatrix}.  
\end{align*}

\end{theorem}

\begin{proof}[{Proof of Theorem \ref{thm_FE_eigen_function_chi_p}}] 
By Lemma \ref{lem_existence_FE_chi_p} it suffices to show the functional equation of the constant term; $\Lambda_{k,\chi_p}^{n,(\nu)}(s) = \Lambda_{k,\chi_p}^{n,(n-\nu)}(\kappa_n-s)$. 
The proof  is quite similar to the proof of Proposition \ref{prop_FE_constant_term_chi_0}. We only consider the case $p \equiv 3 \bmod 4$, equivalently to the case that $k$ is odd.
By multiplying suitable power of $\pi$ and Gamma functions we have
\begin{equation*}  \Lambda^{n,(\nu)}_{k,\chi_0}(s)  = (\sqrt{-1})^{-\nu(k-1)} 2^{\nu(\nu+3)/2-2 \nu s} \pi^{\nu(\nu+1)/4} G_\nu(s) Z_\nu(s) 
\end{equation*}
with
\begin{align*} G_\nu(s) & = \dfrac{\Gamma_\nu \left( 2s-\dfrac{\nu+1}{2} \right)\Gamma_n\left(s + \dfrac{k}{2} \right)}{\Gamma_\nu \left(s+\dfrac{k}{2} \right) \Gamma_\nu\left(s-\dfrac{k}{2}\right)\Gamma_n \left(s+\dfrac{1}{2} \right)} \cdot \dfrac{\Gamma\left( s + \dfrac{1}{2} \right)}{\Gamma\left(s-\dfrac{\nu-1}{2} \right)}  \\  & \quad \times \prod_{i=1}^{[\nu/2]} \frac{\Gamma(2s-i)}{\Gamma \left(2s+i- \dfrac{2\nu+1}{2} \right)},
\end{align*}
\begin{align*} Z_\nu(s) & = \xi(\chi_p, 2s-\nu) \prod_{j=1}^{[\nu/2]}\xi(4s+2j-2\nu-1) \prod_{j=[\nu/2]+1}^{[n/2]} \xi(4s-2j) \\
& \quad \times  (\det Y)^{s-k/2} \zeta^{(n)}_\nu \left(Y, 2s- \frac{\nu+1}{2} \right).
\end{align*}
One can rewrite $Z_\nu(s)$ as (\ref{eq_Z_nu(s)_nu_le_la}) or (\ref{eq_Z_nu(s)_nu_ge_la}) replacing $\xi(n-\nu)$ to $\xi(\chi_p, n-\nu)$; it shows that $Z_\nu(s) = Z_{\la}(\kappa_n-s)$ holds.

Also the term $G_\nu(s)$, by the similar procedure as above we have
\[ G_\nu(s) = \pi^{-\nu(\nu+1)/4} 2^{2s \nu - \nu(\nu+3)/2} \frac{ \Gamma_n\left( s+ \dfrac{k}{2} \right)  \Gamma_\nu \left(s - \dfrac{1}{2} \right) \Gamma_\nu \left( s + \dfrac{1}{2} \right)}{ \Gamma_\nu \left(s+\dfrac{k}{2} \right)\Gamma_\nu \left(s-\dfrac{k}{2} \right)\Gamma_n \left(s + \dfrac{1}{2} \right)}. \]By
\begin{align*} 
& \frac{\Gamma_n\left(s + \dfrac{k}{2} \right)}{\Gamma_\nu \left(s + \dfrac{k}{2} \right)}  =  \pi^{\widetilde{n}-\widetilde{\nu}-\widetilde{\la}} \, \Gamma_\la \left(s -\frac{\nu-1}{2} + \frac{k-1}{2} \right), \\ & \frac{\Gamma_\nu \left(s + \dfrac{1}{2} \right)}{\Gamma_n\left(s + \dfrac{1}{2} \right)}  = \pi^{-\widetilde{n}+ \widetilde{\nu}+\widetilde{\la}} \frac{1}{ \Gamma_\la \left(s - \dfrac{ \nu-1}{2}  \right)},
\end{align*}
the term $G^*_\nu(s) = (\sqrt{-1})^{-\nu(k-1)} 2^{\nu(\nu+3)/2-2 \nu s} \pi^{\nu(\nu+1)/4} G_\nu(s)$ is given by
\[ G^*_\nu(s) = (\sqrt{-1})^{-\nu(k-1)} \frac{\Gamma_\nu\left(s-\dfrac{1}{2} \right)}{\Gamma_\nu \left( s-\dfrac{1}{2} - \dfrac{k-1}{2} \right)} \frac{\Gamma_\la \left(s -\dfrac{\nu-1}{2} + \dfrac{k-1}{2} \right)}{\Gamma_\la \left(s - \dfrac{ \nu-1}{2}  \right)}. \]
Since $k-1$ is even, we apply Lemma \ref{lem_FE_Gamma_function} and get $G^*_\nu(s) = G^*_{\la}(\kappa_n-s)$. This proves the proposition.

\end{proof}

\section{Examples} \label{section_example}

In this section we calculate the example in the case of degree $2$ with quadratic character  $\chi_p$. By Proposition \ref{prop_eigen_Eisenstein_series}  we have
\[ E^2_{k,\chi_p}(w_1;Z,s) = E^{2,(1)}_{k,\chi_p}(Z,s). \]
Let $N \in \Sym^2(\Z)^*$ be positive definite. Then the Fourier coefficient of $E^{2,(1)}(Z,s)$ at $N$ is given by
\[ \det(Y)^{s-k/2} \mathscr{S}^{(1)}_2(\chi_p,N,2s) \xi_2 \left( Y,N, s + \frac{k}{2}, s-\frac{k}{2} \right). \]
By \cite[(4.34K)]{Shi1} we know
\begin{align*}  \xi_2 \left( Y,N, s + \frac{k}{2}, s- \frac{k}{2} \right) & = (-1)^k 2^{2k+2} \pi^{2s+k} \frac{\det(Y)^{k/2-s}}{\Gamma_{2} \left(s+ \dfrac{k}{2} \right)} \\
& \quad \times  \det(N)^{s+k/2-3/2} \w_2 \left(2 \pi Y, N; s+\frac{k}{2}, s-\frac{k}{2} \right). 
\end{align*}
On the other hand $\mathscr{S}_2^{(1)}(\chi_p,N,s)$ is written as follows. Let $D_N$ be the absolute value of the discriminant of the quadratic field $K=\Q(\sqrt{- \det(2N)})$. We write $\det(2N) = D_N \mathfrak{f}^2$ and $f_q = \ord_q \mathfrak{f}$ for each prime $q$. We set $\chi_N$ the quadratic character associated with $K$. that is defined as follows. We have $\chi_N(-1) = -1$ and for a prime number $q$ with $(q,D_N) = 1$, we have $\chi_N(q) = \left( \dfrac{-D_N}{q} \right)$. Here we understand if $q=2$, then
\[ \left( \frac{a}{2} \right)  = \begin{cases} 0 & a \equiv 0 \bmod 2, \\ 1 & a \equiv \pm 1 \bmod 8, \\ -1 & a \equiv \pm 5 \bmod 8. \end{cases} \]
Let $N[\g] = p^m (\alpha, p^t \beta )$ with $\alpha, \beta \in \Z_p^\times$ for some $\g \in \GL_2(\Z_p)$. Then $p \mid D_N$ if and only if $t$ is odd. If $t$ is even then the Dirichlet character $\chi_N \chi_p$ is primitive, on the other hand if $t$ is odd then we have $\chi_N \chi_p = \chi_N' \chi_0$ for a certain primitive character $\chi_N'$. We put $\chi_N^* = \chi_N \chi_p$ or $\chi_N'$ according as $t$ is even or odd respectively.  We put $l_N = f_p-m + \ord_p D_N$, that is $l_N = t/2$ or $(t+1)/2$ according as  $t$ is even or odd respectively. Then by \cite[Hifrssatz 10]{Kau} we can write
\begin{align*} \mathscr{S}_2^{(1)}(\chi_p,N,2s) & = S_2^{(1)}(\chi_p,N,s)_p \frac{L(\chi_N^*, 2s-1)(1-\chi_N^*(p)p^{1-2s})}{L(\chi_p,2s) \zeta(4s-2)(1-p^{2-4s})} \\
& \qquad \times  \prod_{q \ne p} F_N^{(q)}(\chi_p(q)q^{-2s}),
\end{align*}
Here $F_N^{(q)}(X)$ is a certain polynomial in $X$, that equals to $1$ unless $q \mid f_q$.  It satisfies the functional equation
\begin{equation} \label{eq_FE_unramified_Siegel_series_deg_2}
F_N^{(q)}(q^{-3}X^{-1}) = (q^3 X^2)^{-f_q} F_N^{(q)}(X).
\end{equation}
The term $S_2^{(1)}(\chi_p,N,2s)$ is computed in \cite[Proposition 3.2]{Gu}. We have
\begin{align*} S_2^{(1)}(\chi_p,N,2s) & = \ve_p \chi_p(\alpha)p^{(2-2s)m+3/2-2s} \frac{1-p^{2-4s}}{(1-p^{3-4s})(1- \chi_N^*(p) p^{1-2s})} \\
& \quad \times \bigl( 1-p^{(3-4s)l_N} - \chi_N^*(p) p^{1-2s} (1-p^{(3-4s)(l_N-1)}) \bigr).
\end{align*}
Therefore the Fourier coefficient $C(N,s)$ of 
\[ \mathbb{E}^{2,(1)}_{k,\chi_p}(Z,s) = \ve_p p^{-1/2} \frac{\Gamma_2 \left( s + \dfrac{k}{2} \right)}{\Gamma_2 \left( s + \dfrac{\delta_p}{2} \right)} \xi(\chi_p,2s) \xi(4s-2) E^2_{k,\chi_p}(w_1;Z,s) \]
 at $N$ is given by
\begin{align*} & C(N,s)   =  2^{2k+2} \pi^{-s+k+(1-\delta_p)/2} \chi_p(\alpha)  p^{s+(\delta_p-1)/2}  \frac{\Gamma(2s-1)}{\Gamma \left( s+ \dfrac{\delta_p-1}{2} \right) } L(\chi_N^*,2s-1 ) \\
& \times \det(N)^{s-3/2 + k/2}  \w_2 \left(2 \pi Y, N; s+\frac{k}{2}, s-\frac{k}{2} \right)  F_N^{(p)}( p^{-2s}) \prod_{q \ne p} F_N^{(q)}(\chi_p(q)q^{-2s}) \end{align*}
with
\begin{equation}  F_N^{(p)}(p^{-2s}) =  \frac{p^{(2-2s)m+3/2-2s}}{1-p^{3-4s}}\bigl( 1-p^{(3-4s)l_N} - \chi_N^*(p) p^{1-2s} (1-p^{(3-4s)(l_N-1)}) \bigr).  \label{eq_Siegel_series_principal_deg2}
\end{equation}  
We denote the absolute value of the conductor of $\chi_N^*$ by $D_N^*$. Then\[ D_N^* = \begin{cases} p D_N & \text{if $(p, D_N) = 1$,} \\ p^{-1} D_N &  \text{if $p \mid D_N$.} \end{cases} \]
Since $\chi_N^*(-1) = - \chi_p(-1)$ ,we have $\delta_{\chi_N^*} = (1 - \chi_N^*(-1))/2 = 1-\delta_p$.
The completed Dirichlet $L$ function $\xi(\chi_N^*,2s-1)$ is given by
\[ \xi(\chi_N^*,2s-1) = \left( \frac{D_N^*}{\pi} \right)^{s-\delta_p/2} \Gamma \left( s-\dfrac{\delta_p}{2} \right) L(\chi_N^*,2s-1). \]
Since we can show that the Gauss sum of $\chi_N^*$ coincides with $(\sqrt{-1})^{1-\delta_p} \sqrt{D_N^*}$, the functional equation $\xi(\chi_N^*,s) = \xi(\chi_N^*,1-s)$ holds.
For both cases $\delta_p=0$ or $1$ we have
\[ \Gamma\left(s + \frac{\delta_p-1}{2} \right) \Gamma \left( s - \frac{\delta_p}{2} \right) =  \Gamma(s) \Gamma \left( s - \frac{1}{2} \right) = 2^{2-2s} \pi^{1/2} \Gamma(2s-1), \]
thus we can write
\begin{align*} 
& C(N,s)  = \chi_p(\alpha)2^{k+3} \pi^{k-\delta_p}p^{s+(\delta_p-1)/2} (D_N^*)^{-s+\delta_p/2} \det(2N)^{s-3/2+k/2}  \\
& \quad \times \xi(\chi_N^*,2s-1)  \w_2 \left(2 \pi Y, N; s+\frac{k}{2}, s-\frac{k}{2} \right)  F_N^{(p)}( p^{-2s}) \prod_{q \ne p} F_N^{(q)}(\chi_p(q)q^{-2s}).
\end{align*}
Let $d = \ord_p D_N \in \{ 0, 1 \}$. Then
$\ord_p \det(2N) = 2 f_p + b$ and  $D_N^* = D_N' p^{1-d}$ with $(D_N',p) = 1$. Since $\det(2N) = p^{2f_p + d} D_N' \prod_{q \ne p} q^{2f_q}$ we can write by using a constant $A$ that is independent of $s$,
\begin{align*}
C(N,s) & = A \, p^{2(f_p + d)s}  \xi(\chi_N^*,2s-1)   \w_2 \left(2 \pi Y, N; s+\frac{k}{2}, s-\frac{k}{2} \right) \\
& \quad \times F_N^{(p)}( p^{-2s}) \prod_{q \ne p} q^{2s f_q } F_N^{(q)}(\chi_p(q)q^{-2s}).
\end{align*}
By \cite[Theorem 4.2]{Shi1} and (\ref{eq_FE_unramified_Siegel_series_deg_2}), \[ \xi(\chi_N^*,2s-1) \w_2 \left(2 \pi Y, N; s+\frac{k}{2}, s-\frac{k}{2} \right)\prod_{q \ne p} q^{2 sf_q } F_N^{(q)}(\chi_p(q)q^{-2s}) \]
is invariant under $s \mapsto 3/2-s$. 
 Note that $2(f_p+d) = 2m + 2l_N$ and one can check frome (\ref{eq_Siegel_series_principal_deg2}) that $p^{2(m+l_N)s} F_N^{(p)} (p^{-2s})$ is also invariant $s \mapsto 3/2-s$, i.e.
\[ F_N^{(p)}(p^{2s-3}) =  p^{(4s-3)(m+l_N)} F_N^{(p)} (p^{-2s}) \]
holds. As a consequence we have $C(N,s) = C(N, 3/2-s)$, that is the evidence of our functional equation.

We remark that the concrete form of the functional equation (Theorem \ref{thm_FE_eigen_function_chi_p}) is formulated so that the constant terms of the Fourier expansion coincide each other. However we have shown the equality of the Fourier coefficients at $N>0$.

\end{document}